\documentclass[12pt]{article}

\usepackage{amssymb,amsmath,amsfonts,amssymb}
\usepackage{graphics,graphicx,color,verbatim}

\def\@abssec#1{\vspace{.05in}\footnotesize \parindent .2in
{\bf #1. }\ignorespaces}

\graphicspath{{/EPSF/}{Figures/}}
\DeclareGraphicsExtensions{.eps}

\newtheorem{theorem}{Theorem}[section]
\newtheorem{lemma}[theorem]{Lemma}
\newtheorem{proposition}[theorem]{Proposition}

\newtheorem{remark}[theorem]{Remark}

\def \Rm {\mathbb R}

\def \Zm {\mathbb Z}

\newcommand{\eps}{\varepsilon}
\newcommand{\ep}{\epsilon}
\newcommand{\E}{\mathbb E}
\renewcommand{\P}{\mathbb P}
\newcommand{\dsum}{\displaystyle\sum}
\newcommand{\dint}{\displaystyle\int}
\newcommand{\pdr}[2]{\dfrac{\partial{#1}}{\partial{#2}}}

\newcommand{\dr}[2]{\dfrac{d{#1}}{d{#2}}}
\newcommand{\drr}[2]{\dfrac{d^2{#1}}{d{#2}^2}}

 \newcommand{\bh}{\mathbf h}
\newcommand{\bj}{\mathbf j} \newcommand{\bk}{\mathbf k}
 
\newcommand{\bn}{\mathbf n}
\newcommand{\bp}{\mathbf p} \newcommand{\bq}{\mathbf q}
\newcommand{\bt}{\mathbf t}
 
\newcommand{\bx}{\mathbf x} \newcommand{\by}{\mathbf y}
\newcommand{\bz}{\mathbf z} 
\newcommand{\bD}{\mathbf D}

\newcommand{\bzero}{\mathbf 0} \newcommand{\bone}{\mathbf 1}

\newcommand{\bxi}{\boldsymbol \xi}

\newcommand{\bzeta}{\boldsymbol \zeta}

\newcommand{\cout}[1]{}

 \renewcommand{\arraystretch}{1.5}

\title{Central limits and homogenization in random media}

\author{Guillaume Bal \thanks{Department of Applied Physics and 
        Applied Mathematics, Columbia University, 
        New York NY, 10027; gb2030@columbia.edu}}

\begin{document}
 
\maketitle


\begin{abstract}
  We consider the perturbation of elliptic operators of the form
  $P(\bx,\bD)$ by random, rapidly varying, sufficiently mixing,
  potentials of the form $q(\frac{\bx}\eps,\omega)$. We analyze the
  source and spectral problems associated to such operators and show
  that the properly renormalized difference between the perturbed and
  unperturbed solutions may be written asymptotically as $\eps\to0$ as
  explicit Gaussian processes. Such results may be seen as central
  limit corrections to the homogenization (law of large numbers)
  process.  Similar results are derived for more general elliptic
  equations in one dimension of space.
  
  The results are based on the availability of a rapidly converging
  integral formulation for the perturbed solutions and on the use of
  classical central limit results for random processes with
  appropriate mixing conditions.
\end{abstract}
 

\renewcommand{\thefootnote}{\fnsymbol{footnote}}
\renewcommand{\thefootnote}{\arabic{footnote}}

\renewcommand{\arraystretch}{1.1}

\paragraph{keywords:}
  Homogenization, central limit, mixing coefficients, partial
  differential equations with random coefficients, random
  oscillatory integrals.

\paragraph{AMS:} 35R60, 35J05, 35P20, 60H05.



\section{Introduction}
\label{sec:intro}

There are many practical applications of partial differential
equations with coefficients that oscillate at a faster scale than the
scale of the domain on which the equation is solved. In such settings,
it is often necessary to model the rapidly oscillatory coefficients as
random processes, whose properties are known only at a statistical
level. The numerical simulation of the resulting partial differential
equation with random coefficients becomes a daunting task.

Two simplifications are then typically considered. The first
simplification consists in assuming that the coefficients oscillate
very rapidly and replacing the equation with random coefficients by a
homogenized equation with deterministic (effective medium)
coefficients. The homogenization of many linear and nonlinear partial
differential equations with periodic \cite{blp1,Milton-02} and random
highly oscillatory coefficients has been obtained to date
\cite{BMW-JRAM-94,CSW-CPAM-05,JKO-SV-94,Koz-MSb-79,L-PTRF-01,LS-CDPE-05,PV-RF-81}.

The solution to the equation with random equations may also be
interpreted as a functional of an infinite number of random variables
and expanded in polynomial chaoses \cite{CM-AM-47,W-AJM-38}.  A second
simplification consists then in discretizing the randomness in the
coefficients over sufficiently low dimensional subspaces -primarily by
Galerkin projection- so that the partial differential equation with
random coefficients may be fully solved numerically. We refer the
reader to e.g.
\cite{BTZ-CMAME-05,FST-CMAME-05,GS-91,HLRZ-JCP-06,MK-CMAME-05,ZZ-JCP-04}
for references on this active area of research. Such problems, which
are posed in domains of dimension $d+Q$, where $d$ is spatial
dimension and $Q$ the dimension of the random space, are
computationally very intensive, although they have the main advantage
of providing realistic statistical fluctuations of the random
solution, which are absent in the homogenization approximation.

The two aforementioned approaches can hardly be reconciled.
Homogenization arises in a limit where the law of large numbers
applies and the solution becomes asymptotically a deterministic
quantity. The number of random variables describing the random
coefficients thus tends to infinity, a limit that is difficult to
obtain by polynomial chaos-type expansions.

In several practical settings such as e.g. the analysis of geological
basins or the manufacturing of composite materials, one may be
interested in an intermediate situation. We may observe experimental
fluctuations in the random solution which are not accounted for by the
homogenized solution, and yet, may be in the presence of a
sufficiently rich random environment so that full solutions of the
equation with random coefficients may not be feasible. This is the
type of settings that motivate the studies of this paper. Our main
objective will be to characterize the statistical structure of 
the corrector to the homogenized, deterministic, limit. Whereas
the deterministic limit may be seen as a law of large number effect,
we are interested in characterizing the next order term, which
arises as an application of the central limit theory.

In most practical cases of interest, starting with the elliptic
operator $\nabla\cdot a_\eps(\bx,\omega)\nabla$, with $\bx\in
D\subset\Rm^d$ and $\omega\in\Omega$ the space of random realizations,
the calculation of the homogenized tensor is difficult and does not
admit analytic expressions except in very simple cases
\cite{JKO-SV-94}. The amplitude of the corrector to homogenization,
let alone its statistical description, remains largely open. The best
estimates currently available in spatial dimension $d\geq2$ may be
found in \cite{Yur-SM-86}; see also \cite{CK-JCAM-07,EMZ-05},
\cite{CN-EJP-00} for discrete equations, and \cite{B-DCDSB-07} for
applications of such error estimates. Only in one dimension of space
do we have an explicit characterization of the effective diffusion
coefficient and of the corrector \cite{BP-AIHP-04}. Unlike the case of
periodic media, where the corrector is proportional to the size of the
cell of periodicity $\eps$, the random corrector to the homogenized
solution is an explicitly characterized Gaussian process of order
$\sqrt\eps$ when the random coefficient has integrable correlation
\cite{BP-AIHP-04}.  In the case of correlations that are non
integrable and of the form $R(t)\sim t^{-\alpha}$ for some
$0<\alpha<1$, the corrector may be shown to be still an explicitly
characterized Gaussian process, but now of order $\eps^\frac\alpha2$
\cite{BGMP-AA}.

\medskip 

The reason why explicit characterizations of the correctors may be
obtained in \cite{BGMP-AA,BP-AIHP-04} is that the solution to the
heterogeneous elliptic equation may be written explicitly.  Correctors
to homogenization have been obtained in more general settings.  The
analysis of homogenized solutions and central limit correctors to
evolution equations with time dependent randomly varying coefficients
is well known; see e.g.
\cite{BP-SIAP-78,EK-86,FGPS-07,khasminskii,kushner,PSV-SMDS} for
reference on the Markov diffusion approximation and the method of the
perturbed test function. In the context of the one-dimensional
Helmholtz equation, this would correspond to solving the equation on
an interval $(0,a)$ with initial conditions of the form $u_\eps(0)$
and $u'_\eps(0)$ known. The asymptotic limit of boundary value
problems, which corresponds in our example to prescribing $u_\eps(0)$
and $u_\eps(a)$, requires somewhat different mathematical techniques.
We refer the reader to \cite{FM-AAP-94,WF-CPAM-79} for results in the
setting of one-dimensional problems.  Note that in the case of a much
stronger potential, in dimension $d=1$ of the form
$\eps^{-\frac12}q_\eps$ instead of $q_\eps$ in the above Helmholtz
operator, the deterministic homogenization limit no longer holds.
Somewhat surprisingly, the solution of a corresponding evolution
equation still converges to a well identified limit; see
\cite{PP-GAK-06}.

In spatial dimensions two and higher, a methodology to compute the
Gaussian fluctuations for boundary value problems of the form $-\Delta
u_\eps +F(u_\eps,\bx,\frac \bx\eps)=f(\bx)$ was developed in
\cite{FOP-SIAP-82}. An explicit expression for the fluctuations was
obtained and proved for the linear equation $(-\Delta +\lambda+q(\frac
\bx\eps))u_\eps(\bx)=f(\bx)$ in dimension $d=3$. In this paper, we
revisit the problem and generalize it to linear problems of the form
$P(\bx,\bD)u_\eps + q_\eps(\bx)u_\eps = f(\bx)$ with an unperturbed
equation $P(\bx,\bD)u =f$, which admits a Green's function
$G(\bx,\by)$ that is more than square integrable (see
\eqref{eq:intgreen} below).  The prototypical example of interest is
the operator $P(\bx,\bD)=-\nabla\cdot a(\bx)\nabla+q_0(\bx)$ with
sufficiently smooth (deterministic) coefficients $a(\bx)$ and
$q_0(\bx)$ posed on a bounded domain with, say, Dirichlet boundary
conditions, for which the Green's function is more than square
integrable in dimensions $1\leq d\leq 3$.

Under appropriate mixing conditions on the random process
$q_\eps(\bx,\omega)$, we will show that arbitrary spatial moments of
the correctors
\begin{displaymath}
  \Big(\dfrac{u_\eps - u_0}{\eps^{\frac d2}},M\Big),
\end{displaymath}
where $u_\eps$ and $u$ are the solutions to perturbed and unperturbed
equations, respectively, and where $M$ is a smooth function, converge
in distribution to Gaussian random variables, which admit a convenient
and explicit representation as a stochastic integral with respect to a
standard (multi-parameter) Wiener process.  If we denote by $u_1$ the
weak limit of $u_{1\eps}=\eps^{-\frac d2}(u_\eps-u_0)$, we observe,
for $1\leq d\leq 3$, that $\E\{v_{1\eps}^2(\bx,\omega)\}$ converges to
$\E\{u_1^2(\bx,\omega)\}$, where $v_{1\eps}$ is the leading term in
$u_{1\eps}$ up to an error term we prove is of order $O(\eps^d)$ in
$L^1(\Omega \times D)$. This shows that the limiting process $u_1$
captures all the fluctuations of the corrector to homogenization.
This result is in sharp contract to the cases $d\geq4$ and to
homogenization in periodic media in arbitrary dimension, where the
weak limit of the corrector captures a fraction of the energy of that
corrector. We thus see that corrections to homogenization are somewhat
different for $1\leq d\leq 3$ and $d\geq4$ so that the square
integrability of the Green's function is a natural condition in the
framework of homogenization in random media.

\medskip

We obtain similar expressions for the spectral elements of the
perturbed elliptic equation. We find that the correctors to the
eigenvalues and the spatial moments of the correctors to the
corresponding eigenvectors converge in distribution to Gaussian
variables as the correlation length $\eps$ vanishes. In the setting
$d=1$, we obtain similar result for more general elliptic operators of
the form $-\frac d{dx}(a_\eps\frac d{dx})+q_0+q_\eps$ by
appropriate use of harmonic coordinates \cite{Koz-MSb-79}.  The
extension to similar operators in dimension $d\geq2$ remains open.

\medskip

As was already mentioned, the theory developed here allows us to
characterize the statistical properties of the solutions to equations
with random coefficients in the limit where the correlation length
(the scale of the heterogeneities) is small compared to the overall
size of the domain. In many practical problems, it is a good
approximation to the statistical structure of the solution of the
equation and possibly be all that one is interested in.

Asymptotically explicit expressions for the correctors may also find
applications in the testing of numerical algorithms. Several numerical
schemes have been developed to estimate the heterogeneous solution
accurately in the regime of validity of homogenization by using
discretizations with a length scale $h$ that is large compared to the
correlation length: $h\gg\eps$; see e.g.
\cite{ABr-MMS-05,A-SINUM-04,EMZ-05,EHW-SINUM-00,OZ-Arc-06}. A possible
application of the explicit expression for the correctors is to see
whether these algorithms can capture the central limit correction to
the solutions to the random partial differential equations.

An other application concerns the reconstruction of the constitutive
parameters of a differential equation from various measurements, for
instance the reconstruction of the potential in a Helmholtz equation
from spectral measurements \cite{KKL-01,RS-MC92}. In such cases,
reconstruction algorithms provide lower-variance reconstructions when
the cross-correlations are known and used optimally in the inversion;
see e.g. \cite{vogel-SIAM02}. Provided that $q_0$ is the deterministic
quantity that we wish to reconstruct, higher-frequency components that
we may not hope to reconstruct still influence available data.  The
correctors obtained in this paper provide asymptotic estimates for the
cross-correlation of the measured data, which allow us to obtain
lower-variance reconstructions for $q_0$; see \cite{BRe-SIP-08}.

\medskip

An outline for the rest of the paper is as follows. Section
\ref{sec:helmholtz} considers the convergence of the corrector to the
homogenized solution for the Helmholtz equation with source term in
dimensions $1\leq d\leq 3$. The proof is based on showing the rapid
convergence of a Lippman-Schwinger-type integral formulation (see
\eqref{eq:intHelm} below), and on applying central limit theorems to
random oscillatory integrals. The behavior of the oscillatory
integrals is considered in arbitrary dimensions in section
\ref{sec:larger}, where a comparison between homogenization in random
and periodic environments is also considered.  The generalization to a
more general one-dimensional elliptic source problem in detailed in
section \ref{sec:1dbrdypbs}. The results on the correctors obtained
for source problems are then extended to correctors for spectral
problems in section \ref{sec:spectral}. The proof is based on adapting
classical results \cite{Kato66} on the convergence of the spectra of
operators that converge on average in the uniform norm. The results
obtained for the spectral problems are then briefly applied to the
analysis of evolution equations. Some concluding remarks are presented
in section \ref{sec:conclu}.

\section{Correctors for Helmholtz equations}
\label{sec:helmholtz}

Consider an equation of the form:
\begin{equation}
  \label{eq:Helmtype}
  \begin{array}{ll}
  P(\bx,\bD) u_\eps + q_\eps u_\eps =f ,\qquad &\bx\in D\\
  u_\eps =0 \qquad &\bx\in \partial D,
  \end{array}
\end{equation}
where $P(\bx,\bD)$ is a (deterministic) self-adjoint, elliptic,
pseudo-differential operator and $D$ an open bounded domain in
$\Rm^d$. We assume that $P(\bx,\bD)$ is invertible with symmetric and
``more than square integrable'' Green's function. More precisely, we
assume that the equation
\begin{equation}
  \label{eq:Helmunpert}
  \begin{array}{ll}
  P(\bx,\bD) u =f ,\qquad &\bx\in D\\
  u=0 \qquad &\bx\in\partial D,
  \end{array}
\end{equation}
admits a unique solution
\begin{equation}
  \label{eq:solHelmunpert}
  u(\bx) = {\cal G} f (\bx) := \dint_{D} G(\bx,\by) f(\by) d\by,
\end{equation}
and that the real-valued and non-negative (to simplify notation)
symmetric kernel $G(\bx,\by)=G(\by,\bx)$ has more than square
integrable singularities so that
\begin{equation}
  \label{eq:intgreen}
  \bx\mapsto\Big(\dint_D  |G|^{2+\eta}(\bx,\by)d\by\Big)^{\frac{1}{2+\eta}}
   \quad \mbox{ is bounded  on } D \mbox{ for some } \eta>0.
\end{equation}  
The assumption is typically satisfied for operators of the form
$P(\bx,D)=-\nabla\cdot a(\bx) \nabla + \sigma(\bx)$ for $a(\bx)$
uniformly bounded and coercive, $\sigma(\bx)\geq0$, and in dimension
$d\leq3$, with $\eta=+\infty$ when $d=1$ (i.e., the Green's function
is bounded), $\eta<\infty$ for $d=2$, and $\eta<1$ for $d=3$.

Let $\tilde q_\eps(\bx,\omega)=q(\frac{\bx}{\eps},\omega)$ be a mean
zero, (strictly) stationary, process defined on an abstract
probability space $(\Omega,{\cal F},\P)$ \cite{breiman}.  The process
$\tilde q_\eps(\bx,\omega)$ will be modified in the sequel as the
process $q_\eps(\bx,\omega)$ appearing in \eqref{eq:Helmtype} to
ensure that solutions to the Helmholtz equation exist.  We assume that
$q(\bx,\omega)$ has an integrable correlation function:
\begin{equation}
  \label{eq:corrq}
  R(\bx) = \E\{q(\by,\omega)q(\by+\bx,\omega)\},
\end{equation}
where $\E$ is mathematical expectation associated to $\P$. The above
expression is independent of $\by$ by stationarity of the process
$q(\bx,\omega)$.  We also assume that $q(\bx,\omega)$ is strongly
mixing in the following sense. For two Borel sets $A,B\subset \Rm^d$,
we denote by ${\cal F}_A$ and ${\cal F}_B$ the sub-$\sigma$ algebras
of ${\cal F}$ generated by the field $q(\bx,\omega)$. Then we assume
the existence of a ($\rho-$) mixing coefficient $\varphi(r)$ such that
\begin{equation}
  \label{eq:strongmixing}
  \Big| \dfrac{\E\big\{(\eta-\E\{\eta\})(\xi-\E\{\xi\})\big\}}
     {\big(\E\{\eta^2\}\E\{\xi^2\}\big)^{\frac12}}\Big|
   \leq \varphi\big(2\,d(A,B)\big)
\end{equation}
for all (real-valued) random variables $\eta$ on $(\Omega,{\cal
  F}_A,\P)$ and $\xi$ on $(\Omega,{\cal F}_B,\P)$.  Here, $d(A,B)$ is
the Euclidean distance between the Borel sets $A$ and $B$. The
multiplicative factor $2$ in \eqref{eq:strongmixing} is here only for
convenience.  Moreover, we assume that $\varphi(r)$ is bounded and
decreasing.  We will impose additional restrictions on the process to
ensure that the equation \eqref{eq:Helmtype} admits a solution.

We formally recast \eqref{eq:Helmtype} as 
\begin{equation}
  \label{eq:ueps1}
  u_\eps  = {\cal G} (f-q_\eps u_\eps),
\end{equation}
where ${\cal G}=P(\bx,D)^{-1}$, and after one more iteration as
\begin{equation}
  \label{eq:intHelm}
  u_\eps = {\cal G} f - {\cal G} q_\eps {\cal G} f + {\cal G} q_\eps
    {\cal G} q_\eps u_\eps.
\end{equation}
This is the integral equation we aim to analyze. 

\subsection{Existence and error estimates}
\label{sec:existerror}

In order for the above equation to admit a unique solution, we need to
ensure that $(I-{\cal G} q_\eps {\cal G} q_\eps)$ is invertible
$\P-$a.s. We modify the process $\tilde q_\eps(\bx,\omega)$ defined
above on a set of measure of order $\eps$ so that ${\cal G} q_\eps
{\cal G} q_\eps$ has spectral radius bounded by $\rho<1$ $\P-$a.s. To
do so and to estimate the source terms in \eqref{eq:intHelm}, we need a
few lemmas.
\begin{lemma}
  \label{lem:mixingfourpoints}
  Let $q(\bx,\omega)$ be strongly mixing so that
  \eqref{eq:strongmixing} holds and such that $\E\{q^6\}<\infty$.
  Then, we have:
  \begin{equation}
    \label{eq:bdq1234}
    \big|\E\{q(\bx_1)q(\bx_2)q(\bx_3)q(\bx_4)\}\big| \lesssim
    \sup\limits_{\{\by_k\}_{1\leq k\leq 4}=\{\bx_k\}_{1\leq k\leq 4}}
    \varphi^{\frac12}(|\by_1-\by_3|)\varphi^{\frac12}(|\by_2-\by_4|)
    \E\{q^6\}^{\frac23}.
  \end{equation}
\end{lemma}
Here, we use the notation $a\lesssim b$ when there is a positive constant $C$
such that $a\leq Cb$.
\begin{proof}
  Let $\by_1$ and $\by_2$ be two points in $\{\bx_k\}_{1\leq k\leq 4}$
  such that $d(\by_1,\by_2)\geq d(\bx_i,\bx_j)$ for all $1\leq
  i,j\leq4$ and such that $d(\by_1,\{\bz_3,\bz_4\})\leq
  d(\by_2,\{\bz_3,\bz_4\})$, where
  $\{\by_1,\by_2,\bz_3,\bz_4\}=\{\bx_k\}_{1\leq k\leq 4}$.
  
  Let us call $\by_3$ a point in $\{\bz_3,\bz_4\}$ closest to $\by_1$.
  We call $\by_4$ the remaining point in $\{\bx_k\}_{1\leq k\leq 4}$.  
  We have, using \eqref{eq:strongmixing} and $\E\{q\}=0$, that:
  \begin{displaymath}
  \big|\E\{q(\bx_1)q(\bx_2)q(\bx_3)q(\bx_4)\}\big| \lesssim
   \varphi(2|\by_1-\by_3|) (\E\{q^2\})^{\frac12}
    \big(\E\{ (q(\by_2)q(\by_3)q(\by_4))^2\}\big)^{\frac12}.
  \end{displaymath}
  The last two terms are bounded by $\E\{q^6\}^{\frac16}$ and
  $\E\{q^6\}^{\frac12}$, respectively, using H\"older's inequality.
  Because $\varphi(r)$ is assumed to be decreasing, we deduce that
  \begin{equation}
    \label{eq:intermqs0}
     \big|\E\{q(\bx_1)q(\bx_2)q(\bx_3)q(\bx_4)\} \big|\lesssim 
      \varphi(|\by_1-\by_3|) \E\{q^6\}^{\frac23}.
  \end{equation}
  
  If $\by_4$ is (one of) the closest point(s) to $\by_2$, then the
  same arguments show that 
  \begin{equation}
    \label{eq:intermqs}
  \big|\E\{q(\bx_1)q(\bx_2)q(\bx_3)q(\bx_4)\}\big| \lesssim
   \varphi(|\by_2-\by_4|)  \E\{q^6\}^{\frac23}.
  \end{equation}

  Otherwise, $\by_3$ is the closest point to $\by_2$, and we find that 
  \begin{displaymath}
    \big|\E\{q(\bx_1)q(\bx_2)q(\bx_3)q(\bx_4)\}\big| \lesssim
   \varphi(2|\by_2-\by_3|)  \E\{q^6\}^{\frac23}.
  \end{displaymath}
  However, by construction, $|\by_2-\by_4|\leq |\by_1-\by_2|\leq
  |\by_1-\by_3|+|\by_3-\by_2|\leq 2 |\by_2-\by_3|$, so
  \eqref{eq:intermqs} is still valid (this is the only place where the
  factor $2$ in \eqref{eq:strongmixing} is used).
  
  Combining \eqref{eq:intermqs0} and \eqref{eq:intermqs}, the result
  follows from $a\wedge b \leq (ab)^{\frac12}$ for $a,b\geq0$, where
  $a\wedge b=\min(a,b)$.
\end{proof}
\begin{lemma}
  \label{lem:boundgqgq}
  Let $q_\eps$ be a stationary process
  $q_\eps(\bx,\omega)=q(\frac{\bx}{\eps},\omega)$ with integrable
  correlation function in \eqref{eq:corrq}.  Let $f$ be a
  deterministic square integrable function on $D$. Then we have:
  \begin{equation}
    \label{eq:bdgqgf}
    \E\{ \|{\cal G}q_\eps{\cal G}f \|^2_{L^2(D)}\} 
   \lesssim \eps^d \|f\|^2_{L^2(D)}.
  \end{equation}

  Let $q_\eps$ satisfy one of the following additional hypotheses:
  \begin{itemize}
  \item[{\rm [H1]}] $q(\bx,\omega)$ is uniformly bounded $\P$-a.s.
  \item[{\rm [H2]}] $\E\{q^6\}<\infty$ and $q(\bx,\omega)$
    is strongly mixing with mixing coefficient in
    \eqref{eq:strongmixing} such that $\varphi^{\frac12}(r)$ is
    bounded and $r^{d-1}\varphi^{\frac12}(r)$ is integrable on
    $\Rm^+$.
  \end{itemize}
  Then we find that 
  \begin{equation}
    \label{eq:bdgqgq}
    \E\{ \|{\cal G}q_\eps{\cal G}q_\eps \|^2_{{\cal L}(L^2(D))}\} 
   \lesssim \eps^d.
  \end{equation}
\end{lemma}
\begin{remark}
  \label{rem:varphi} \rm
  Note the assumption [H2] combined with $\varphi(r)$ decreasing
  together impose that $\varphi(r)=o(r^{-2d})$. For otherwise, we
  would have an increasing sequence $r_n\to\infty$ as $n\to\infty$
  such that $\varphi^{\frac12}(r_n)\geq C r_n^{-d}$ for some $C>0$,
  and then, since $\varphi^{\frac12}$ is also decreasing,
  \begin{displaymath}
    \dint_0^\infty r^{d-1}\varphi^{\frac12}(r) dr \geq\dsum_n
   \dint_{r_n}^{r_{n+1}} \dfrac{r^{d-1}dr}{r_{n+1}^d} 
   = \dsum_n \dfrac{r_{n+1}^d-r_n^d}{dr_{n+1}^d}
    \geq \dsum_n \dfrac{r_{n+1}-r_n}{dr_{n+1}}.
  \end{displaymath}
  Now if there is an infinite number of terms $n$ such that
  $r_{n+1}\geq2r_n$, then there is an infinite number of terms such
  that $\frac{r_{n+1}-r_n}{dr_{n+1}}\geq \frac1{2d}$ and the above sum
  is infinite. If there is a finite number of such terms, then for all
  $n\geq n_0$ for $n_0$ sufficiently large, we have $r_{n+1}\leq2r_n$
  so that
  \begin{displaymath}
    \dint_0^\infty r^{d-1}\varphi^{\frac12}(r) dr 
     \geq \dsum_{n\geq n_0} \dfrac{r_{n+1}-r_n}{dr_{n+1}}
     \geq \dsum_{n\geq n_0} \dfrac{r_{n+1}-r_n}{2dr_{n}}
     \geq \dfrac{1}{2d}\dint_{r_{n_0}}^\infty \dfrac{dx}x = +\infty.
  \end{displaymath}
  Our assumptions then impose that $\varphi(r)$ decay faster than
  $r^{-2d}$.
\end{remark}
\begin{proofof}[Lemma \ref{lem:boundgqgq}].
  We denote $\|\cdot\| = \|\cdot\|_{L^2(D)}$ and calculate
  \begin{displaymath}
    {\cal G} q_\eps {\cal G}f (\bx)
   = \dint_{D} \Big(\dint_{D} 
      G(\bx,\by) q_\eps(\by) G(\by,\bz) d\by\Big)  f(\bz) d\bz,
  \end{displaymath}
  so that by the Cauchy-Schwarz inequality, we have
  \begin{displaymath}
    |{\cal G} q_\eps {\cal G}f (\bx)|^2
    \leq \|f\|^2 \dint_{D} \Big( \dint_{D} 
      G(\bx,\by) q_\eps(\by) G(\by,\bz) d\by\Big)^2 d\bz .
  \end{displaymath}
  By definition of the correlation function, we thus find that
  \begin{equation}
    \label{eq:integralcorr0}
    \E\{ \|{\cal G}q_\eps{\cal G}f \|^2\} \lesssim \|f\|^2
   \dint_{D^4} G(\bx,\by) G(\bx,\bzeta)
      R\Big(\dfrac{\by-\bzeta}\eps\Big) G(\by,\bz)G(\bzeta,\bz) 
     d\bx d\by d\bzeta d\bz.
  \end{equation}
  Extending $G(\bx,\by)$ by $0$ outside $D\times D$, we find in the 
  Fourier domain that 
  \begin{displaymath}
    \E\{ \|{\cal G}q_\eps{\cal G}f \|^2\} \lesssim \|f\|^2 \dint_{D^2}
    \dint_{\Rm^d} |\widehat{G(\bx,\cdot)G(\bz,\cdot)}|^2(\bp)  \eps^d 
    \hat R(\eps\bp) d\bp d\bx d\bz.
  \end{displaymath}
  Here $\hat f(\bxi)=\int_{\Rm^d}e^{-i\bxi\cdot\bx}f(\bx)d\bx$ is the
  Fourier transform of $f(\bx)$.  Since $R(\bx)$ is integrable, then
  $\hat R(\eps\bp)$ (which is always non-negative by e.g. Bochner's
  theorem) is bounded by a constant we call $R_0$ so that
  \begin{displaymath}
    \E\{ \|{\cal G}q_\eps{\cal G}f \|^2\} \lesssim \|f\|^2 \eps^d R_0
     \dint_{D^3} G^2(\bx,\by)G^2(\bz,\by) d\bx d\by d\bz
     \lesssim  \|f\|^2 \eps^d R_0,
  \end{displaymath}
  by the square-integrability assumption on $G(\bx,\by)$.  This yields
  \eqref{eq:bdgqgf}.  Let us now consider \eqref{eq:bdgqgq}. We denote
  by $\|{\cal G} q_\eps {\cal G} q_\eps\|$ the norm $\|{\cal
    G}q_\eps{\cal G}q_\eps \|_{{\cal L}(L^2(D))}$ and calculate that
  \begin{displaymath}
    {\cal G} q_\eps {\cal G} q_\eps \phi (\bx)
   = \dint_{D} \Big(\dint_{D} 
      G(\bx,\by) q_\eps(\by) G(\by,\bz) d\by\Big)  q_\eps(\bz) \phi(\bz) d\bz.
  \end{displaymath}
  Therefore,
  \begin{displaymath}
    \Big({\cal G} q_\eps {\cal G} q_\eps \phi (\bx)\Big)^2  \leq
   \dint_D \Big(\dint_D 
      G(\bx,\by) q_\eps(\by) G(\by,\bz)q_\eps(\bz)  d\by\Big) ^2  d\bz \dint_D
      \phi^2(\bz)d\bz,  
  \end{displaymath}
  by Cauchy Schwarz. This shows that 
  \begin{displaymath}
    \|{\cal G} q_\eps {\cal G} q_\eps\|^2(\omega) \leq 
   \dint_{D^2} \Big(\dint_{D} 
      G(\bx,\by) q_\eps(\by) G(\by,\bz)  d\by\Big) ^2 q^2_\eps(\bz) d\bz d\bx.
  \end{displaymath}
  When $q_\eps(\bz,\omega)$ is bounded $\P-$a.s., the above proof
  leading to \eqref{eq:bdgqgf} applies and we obtain
  \eqref{eq:bdgqgq} under hypothesis [H1]. 

  Using Lemma \ref{lem:mixingfourpoints}, we obtain that
  \begin{displaymath}
    \E \{ q_\eps(\by)q_\eps(\bzeta)q_\eps^2(\bz)\}
    \lesssim \varphi^{\frac12}\Big(\dfrac{|\by-\bzeta|}{\eps}\Big)
    \varphi^{\frac12}(0)  + 
    \varphi^{\frac12}\Big(\dfrac{|\by-\bz|}{\eps}\Big)
    \varphi^{\frac12}\Big(\dfrac{|\bz-\bzeta|}{\eps}\Big).
  \end{displaymath}
  Under hypothesis [H2], we thus obtain that 
  \begin{displaymath}
   \begin{array}{rcl}
    \E \{\|{\cal G} q_\eps {\cal G} q_\eps\|^2 \}
   &\lesssim &\dint_{D^4} G(\bx,\by) G(\bx,\bzeta)
      \varphi^{\frac12}
   \Big(\dfrac{|\by-\bzeta|}\eps\Big) G(\by,\bz)G(\bzeta,\bz) d\by d\bzeta
    d\bx d\bz \\
    &&+ \dint_{D^2} \Big(\dint_D G(\bx,\by) 
    \varphi^{\frac12}\Big(\dfrac{|\by-\bz|}{\eps}\Big)
     G(\by,\bz)d\by\Big)^2 d\bx d\bz.
   \end{array}
  \end{displaymath}
  Because $r^{d-1}\varphi^{\frac12}(r)$ is integrable, then
  $\bx\mapsto\varphi^{\frac12}(|\bx|)$ is integrable as well and the
  bound of the first term above under hypothesis [H2] is done as in
  \eqref{eq:integralcorr0} by replacing $R(\bx)$ by
  $\varphi^{\frac12}(|\bx|)$. As for the second term, it is bounded,
  using the Cauchy Schwarz inequality, by
  \begin{displaymath}
    \dint_{D} \Big(\dint_{D} \Big( \dint_DG^2(\bx,\by)d\bx\Big)
     G^2(\by,\bz) d\by\Big) 
    \Big( \dint_{D}  \varphi
   \Big(\dfrac{|\by-\bz|}\eps\Big) d\by \Big)d\bz \lesssim \eps^d,
  \end{displaymath}
  since $\bx\mapsto\varphi(|\bx|)$ is integrable, $D$ is bounded, and
  \eqref{eq:intgreen} holds.
\end{proofof}
Applying the previous result to the process $\tilde
q_\eps(\bx,\omega)=q(\frac{\bx}\eps,\omega)$, we obtain from the
Chebyshev inequality that
\begin{equation}
  \label{eq:pbigradius}
  \P(\omega; \|{\cal G}\tilde q_\eps {\cal G}\tilde q_\eps\|>\rho)
   \lesssim \dfrac{\E\{\|{\cal G}\tilde q_\eps {\cal G}\tilde q_\eps\|^2\}}
  {\rho^2} \lesssim \eps^d.
\end{equation}
On the domain $\Omega_\eps\subset\Omega$ of measure
$\P(\Omega_\eps)\lesssim\eps^d$ where $\|{\cal G} \tilde q_\eps {\cal
  G} \tilde q_\eps\|>\rho$, we modify the potential $\tilde q_\eps$
and set it to e.g. $0$. We thus construct
\begin{equation}
  \label{eq:modpotqeps}
  q_\eps(\bx,\omega) = \left\{
    \begin{array}{ll}
        \tilde q_\eps(\bx,\omega) & \Omega\backslash\Omega_\eps, \\
        0 & \Omega_\eps.
    \end{array}\right.
\end{equation}
We have
\begin{lemma}
  \label{lem:qeps} The results obtained for 
  $\tilde q_\eps(\bx,\omega)=q(\frac{\bx}{\eps},\omega)$ in Lemma
  \ref{lem:boundgqgq} hold for $q_\eps(\bx,\omega)$ constructed in
  \eqref{eq:modpotqeps}.
\end{lemma} 
\begin{proof}
  For instance,
  \begin{displaymath}
   \begin{array}{l}
    \E\{ \|{\cal G}q_\eps{\cal G}f \|^2 \}
    =\E\{ \chi_{\Omega^\eps}(\omega)\|{\cal G}q_\eps{\cal G}f \|^2\}
    + \E\{ \chi_{\Omega\backslash\Omega^\eps}(\omega)
     \|{\cal G}q_\eps{\cal G}f \|^2\}\\[2mm]
    = \E\{ \chi_{\Omega\backslash\Omega^\eps}(\omega)
     \|{\cal G}\tilde q_\eps{\cal G}f \|^2\} \leq
     \E\{\|{\cal G}\tilde q_\eps{\cal G}f \|^2\}\lesssim \eps^d \|f\|^2 .
   \end{array}
  \end{displaymath}
  The same proof holds for the second bound \eqref{eq:bdgqgq}.
\end{proof}
We also need to assume that the oscillatory integrals studied in
subsequent sections are not significantly modified when
$q(\frac{\bx}\eps,\omega)$ is replaced by the new
$q_\eps(\bx,\omega)$. We assume that
\begin{itemize}
\item[{\rm [H3]}] \hfill
  \begin{math}
   \lim\limits_{\eps\to0} \E \Big\{ \Big\| \dfrac{1}{\eps^{\frac d2}}
    \Big(q\Big(\dfrac{\bx}{\eps},\omega\Big)
   - q_\eps(\bx,\omega)\Big)\Big\| \Big\}
   \equiv\lim\limits_{\eps\to0} \E \Big\{ \chi_{\Omega^\eps}(\omega)
    \Big\|\dfrac{1}{\eps^{\frac d2}}
  q\Big(\dfrac{\bx}{\eps},\omega\Big)\Big\| \Big\}  =0.
  \end{math} \hfill $\mbox{}$
\end{itemize}
Note that such a condition is automatically satisfied when
$\eps^{-\alpha d}q(\bx,\omega)$ is bounded $\P$-a.s for $0\leq
\alpha<\frac 12$.

With the modified potential, \eqref{eq:intHelm} admits a unique
solution $\P$-a.s. and we find that 
\begin{equation}
  \label{eq:bduepsf}
  \|u_\eps\|(\omega) \lesssim \|{\cal G}f\| 
    + \|{\cal G} q_\eps {\cal G} f\| \quad \P-\mbox{ a.s.},
\end{equation}
where $\|\cdot\|$ denotes $L^2(D)$ norm. 
Using the first result of Lemma \ref{lem:boundgqgq}, we find that 
\begin{equation}
  \label{eq:l2bdueps}
  \E \{ \|u_\eps\|^2 \} \lesssim \|f\|^2.
\end{equation}

Now we can address the behavior of the correctors.  We define
\begin{equation}
  \label{eq:homog}
  u_0= {\cal G} f,
\end{equation}
the solution of the unperturbed problem. We find that 
\begin{equation}
   \label{eq:expansioncorr}
  (I-{\cal G}q_\eps {\cal G}q_\eps) (u_\eps - u_0) 
   = -{\cal G}q_\eps {\cal G} f +
    {\cal G}q_\eps {\cal G}q_\eps {\cal G} f.  
\end{equation}
Using the results of Lemma \ref{lem:boundgqgq}, we obtain that 
\begin{lemma}
  \label{lem:estcorrHelm}
  Let $u_\eps$ be the solution to the heterogeneous problem
  \eqref{eq:Helmtype} and $u_0$ the solution to the corresponding
  homogenized problem. Then we have that
  \begin{equation}
  \label{eq:bduepsmu0}
  \big(\E \{ \| u_\eps - u_0\|^2 \}\big)^{\frac12} \
    \lesssim \eps^{\frac d2} \|f\|.
\end{equation}
\end{lemma}
Note that if we write $u_\eps=A_\eps f$ and $u_0=A_0 f$, with $A_\eps$
and $A_0$ the solution operators of the heterogeneous and homogeneous
equations, respectively, then we have just shown that
\begin{equation}
  \label{eq:bdaepshelm}
  \E\{\|A_\eps-A_0\|^2\} \lesssim \eps^{d}.
\end{equation}

Now ${\cal G}q_\eps {\cal G}q_\eps(u_\eps - u_0)$ is bounded by
$\eps^d$ in $L^1(\Omega;L^2(D))$ by Cauchy-Schwarz:
\begin{displaymath}
  \E\{\|{\cal G}q_\eps {\cal G}q_\eps(u_\eps - u_0)\|\}
  \leq \Big( \E\{\|{\cal G}q_\eps {\cal G}q_\eps\|^2\}\Big)^{\frac12}
   \Big( \E\{\|u_\eps - u_0 \|^2\}\Big)^{\frac12} \lesssim \eps^d.
\end{displaymath}

We need the following estimate:
\begin{lemma}
  \label{lem:gqgqgf}
  Under hypothesis {\rm [H2]} of Lemma \ref{lem:boundgqgq}, we find that 
  \begin{equation}
    \label{eq:estimgqgqgf}
    \E\{\|{\cal G}q_\eps{\cal G}q_\eps{\cal G}f\|^2\}
     \lesssim \eps^{2d\frac{1+\eta}{2+\eta}} \|f\|^2
     \ll \eps^d \|f\|^2,
  \end{equation}
  where $\eta$ is such that $\by\mapsto \Big(\dint_D
  |G|^{2+\eta}(\bx,\by)d\bx\Big)^{\frac1{2+\eta}}$ is uniformly
  bounded on $D$.
\end{lemma}
\begin{proof}
  By Cauchy Schwarz,
  \begin{displaymath}
    |{\cal G}q_\eps{\cal G}q_\eps{\cal G}f (\bx)|^2
    \leq \|f\|^2 \dint_D \Big(\dint_{D^2} G(\bx,\by) q_\eps(\by)G(\by,\bz)
    q_\eps(\bz) G(\bz,\bt)d\by d\bz\Big)^2 d\bt.
  \end{displaymath}
  So we want to estimate
  \begin{displaymath}
    A=\E\{ \dint_{D^6} G(\bx,\by)G(\bx,\bzeta)q_\eps(\by)q_\eps(\bzeta)
    G(\by,\bz)G(\bzeta,\bxi) q_\eps(\bz)q_\eps(\bxi)
    G(\bz,\bt) G(\bxi,\bt) d[\bxi\bzeta\by\bz\bx\bt]\}.
  \end{displaymath}
  We now use \eqref{eq:bdq1234} to obtain that $A\lesssim
  A_1+A_2+A_3$, where
  \begin{displaymath}
   \begin{array}{l}
   A_1=\dint_{D^6} G(\bx,\by)G(\bx,\bzeta) \varphi^{\frac12}
    \Big(\dfrac{|\by-\bzeta|}{\eps}\Big)
    G(\by,\bz)G(\bzeta,\bxi) \varphi^{\frac12}
    \Big(\dfrac{|\bz-\bxi|}{\eps}\Big)
    G(\bz,\bt) G(\bxi,\bt) d[\bxi\bzeta\by\bz\bx\bt],\\
   A_2=\dint_{D^2} \Big(\dint_{D^2} G(\bx,\by) G(\by,\bz)
      \varphi^{\frac12} \Big(\dfrac{|\by-\bz|}{\eps}\Big)
    G(\bz,\bt)d\by d\bz\Big)^2 d\bt d\bx,\\
   A_3=\dint_{D^6} G(\bx,\by)G(\bxi,\bt)G(\bx,\bzeta) G(\bz,\bt)
    \varphi^{\frac12}\Big(\dfrac{|\by-\bxi|}{\eps}\Big)
    G(\by,\bz)G(\bzeta,\bxi) \varphi^{\frac12}
    \Big(\dfrac{|\bzeta-\bz|}{\eps}\Big)
      d[\bxi\bzeta\by\bz\bx\bt].
   \end{array}
  \end{displaymath}
  Denote $F_{\bx,\bt}(\by,\bz)=G(\bx,\by)G(\by,\bz)G(\bz,\bt)$. Then in the
  Fourier domain, we find that 
  \begin{displaymath}
    A_1\lesssim \dint_{D^2}\dint_{\Rm^{2d}}
     \eps^{2d} \widehat{\varphi^{\frac12}}(\eps\bp)
     \widehat{\varphi^{\frac12}}(\eps\bq) 
     |\hat F_{\bx,\bt}(\bp,\bq)|^2 d\bp d\bq
    d\bx d\bt.
  \end{displaymath}
  Here $\widehat{\varphi^{\frac12}}(\bp)$ is the Fourier transform of
  $\bx\mapsto\varphi^{\frac12}(|\bx|)$.  Since
  $\widehat{\varphi^{\frac12}}(\eps\bp)$ is bounded because
  $r^{d-1}\varphi^{\frac12}(r)$ is integrable on $\Rm^+$, we deduce
  that
  \begin{displaymath}
    A_1 \lesssim \eps^{2d} \dint_{D^4}
    G^2(\bx,\by)G^2(\by,\bz)G^2(\bz,\bt) d\bx d\by d\bz d\bt
    \lesssim \eps^{2d},
  \end{displaymath}
  using the integrability condition imposed on $G(\bx,\by)$.

  Using $2ab\leq a^2+b^2$ for $(a,b)=(G(\bx,\by),G(\bx,\bzeta))$ and
  $(a,b)=(G(\bxi,\bt),G(\bz,\bt))$ successively, and integrating
  in $\bt$ and $\bx$, we find that
  \begin{displaymath}
    A_3 \lesssim \dint_{D^4}G(\by,\bz)G(\bzeta,\bxi)
       \varphi^{\frac12}\Big(\dfrac{|\by-\bxi|}{\eps}\Big)
   \varphi^{\frac12} \Big(\dfrac{|\bzeta-\bz|}{\eps}\Big) 
    d[\by\bzeta\bz\bxi],
  \end{displaymath}
  thanks to \eqref{eq:intgreen}. Now with
  $(a,b)=(G(\by,\bz),G(\bzeta,\bxi))$, we find that
  \begin{displaymath}
    A_3 \lesssim \dint_{D^4} G^2(\by,\bz)
    \varphi^{\frac12}\Big(\dfrac{|\by-\bxi|}{\eps}\Big)
   \varphi^{\frac12} \Big(\dfrac{|\bzeta-\bz|}{\eps}\Big) 
    d[\by\bzeta\bz\bxi] \lesssim \eps^{2d},
  \end{displaymath}
  since $\varphi^{\frac12}$ is integrable and $G$ is square integrable
  on the bounded domain $D$.

  
  Let us now consider the contribution $A_2$. We write the squared
  integral as a double integral over the variables
  $(\by,\bzeta,\bz,\bxi)$ and dealing with the integration in $\bx$
  and $\bt$ using $2ab\leq a^2+b^2$ as in the $A_3$ contribution,
  obtain that
  \begin{displaymath}
    A_2 \lesssim \dint_{D^4} 
    G(\by,\bzeta)\varphi^{\frac12}\Big(\dfrac{|\by-\bzeta|}{\eps}\Big)
    G(\bz,\bxi)\varphi^{\frac12}\Big(\dfrac{|\bz-\bxi|}{\eps}\Big)
    d[\by\bzeta\bz\bxi].
  \end{displaymath}
  Using H\"older's inequality, we obtain that
  \begin{displaymath}
    A_2 \lesssim \Big(\Big(
    \dint_0^\infty \varphi^{\frac{p'}2} \Big(\dfrac{r}{\eps}\Big)r^{d-1}
     dr\Big)^{\frac{1}{p'}}
    \Big(\dint_{D^2} G^p(\by,\bz) d\by d\bz \Big)^{\frac 1p}\Big)^2
    \lesssim \eps^{2d\frac{1+\eta}{2+\eta}},
  \end{displaymath}
  with $p=2+\eta$ and $p'=\frac{2+\eta}{1+\eta}$ since
  $\varphi^{\frac12}(r)r^{d-1}$, whence
  $\varphi^{\frac{p'}2}(r)r^{d-1}$, is integrable.
\end{proof}
The above lemma applies to the stationary process $\tilde
q_\eps(\bx,\omega)$, and using the same proof as in Lemma
\ref{lem:qeps}, for the modified process $q_\eps(\bx,\omega)$ in
\eqref{eq:modpotqeps}.  We have therefore obtained that
\begin{equation}
  \label{eq:bduuepscorr}
  \E\{\|u_\eps-u+{\cal G}q_\eps{\cal G}f \|\} \lesssim \eps^d.
\end{equation}
For what follows, it is useful to recast the above result as:
\begin{proposition}
  \label{prop:strongcv}
  Let $q(\bx,\omega)$ be constructed so that [H2]-[H3] holds and let
  $q_\eps(\bx,\omega)$ be as defined in \eqref{eq:modpotqeps}. Let
  $u_\eps$ be the solution to \eqref{eq:intHelm} and $u_0={\cal G}f$.
  We assume that $u_0$ is continuous on $D$. Then we have the
  following strong convergence result:
  \begin{equation}
    \label{eq:limuuepscorr}
    \lim\limits_{\eps\to0} \E\Big\{\Big\|\dfrac{u_\eps-u_0}{\eps^{\frac d2}}
    + \dfrac{1}{\eps^{\frac d2}} {\cal G}q\Big(\dfrac{\cdot}\eps,\omega\Big)
    u_0 \Big\|\Big\} =0.
  \end{equation}
\end{proposition}
\begin{proof}
  Thanks to hypothesis [H3], we may replace $q_\eps(\bx,\omega)$ by
  $\tilde q_\eps(\bx,\omega)=q(\frac{\bx}\eps,\omega)$ in
  \eqref{eq:bduuepscorr} up to a small error compared to $\eps^{\frac
    d2}$. Indeed,
  \begin{displaymath}
    \begin{array}{l}
    \E\Big\{  \Big\| \dfrac{1}{\eps^{\frac d2}} {\cal G} 
     \Big( q\Big(\dfrac{\cdot}\eps,\omega\Big)- q_\eps(\cdot,\omega)\Big)
    u_0 \Big\|\Big\} = \E \Big\{  \chi_{\Omega^\eps}(\omega)
    \Big\|\dfrac{1}{\eps^{\frac d2}} {\cal G} 
     q\Big(\dfrac{\cdot}\eps,\omega\Big) u_0 \Big\|\Big\}\\
    \leq \|{\cal G} \| \|u_0\|_{L^\infty(D)} 
    \E \Big\{ \chi_{\Omega^\eps}(\omega) 
     \Big\|\dfrac{1}{\eps^{\frac d2}} q\Big(\dfrac{\cdot}\eps,\omega\Big)
      \Big\|\Big\} \ll 1.
     \end{array}
  \end{displaymath}
\end{proof}
The rescaled corrector $\eps^{-\frac d2}{\cal
  G}q(\frac{\cdot}\eps,\omega) u_0$ does not converge strongly to its
limit. Rather, it should be interpreted as a stochastic oscillatory
integral whose limiting distribution is governed by the central limit
theorem \cite{breiman,EK-86}. We consider such limits first in the
one-dimensional case and second for arbitrary space dimensions.

\subsection{Oscillatory integral in one space dimension}
\label{sec:helm1d}

In dimension $d=1$, the leading term of the corrector
$\eps^{-\frac12}(u_\eps-u_0)$ is thus given by:
\begin{equation}
  \label{eq:u1eps1dhelm}
  u_{1\eps}(x,\omega)=
  \dint_D -G(x,y)  \dfrac{1}{\sqrt\eps}q(\frac y\eps,\omega) u_0(y) dy,
\end{equation}
where $D$ is an interval $(a,b)$. The convergence is more precise in
dimension $d=1$ than in higher space dimensions. For the Helmholtz
equation, the Green function in $d=1$ is Lipschitz continuous and we
will assume this regularity for the rest of the section; see the next
section for less regular Green's functions. Then $u_{1\eps}(x,\omega)$
is of class ${\cal C}(D)$ $\P$-a.s. and we can seek convergence in
that functional class. Since $u_0={\cal G}f$, it is continuous
for $f\in L^2(D)$.

The variance of the random variable $u_{1\eps}(x,\omega)$ is given by
\begin{equation}
  \label{eq:varu11dhelm}
  \E\{u_{1\eps}^2(x,\omega)\} = \dint_{D^2} G(x,y)G(x,z) 
   \dfrac 1\eps R\Big(\dfrac{y-z}\eps\Big) u_0(y)u_0(z) dydz.
\end{equation}
Because $R(x)$ is assumed to be integrable, the above integral
converges, as $\eps\to0$, to the following limit:
\begin{equation}
  \label{eq:varu1lim1dhelm}
  \E\{u_1^2(x,\omega)\} = \dint_D G^2(x,y) \hat R(0) u_0^2(y) dy,
\end{equation}
where 
\begin{equation}
  \label{eq:hatR0}
  \hat R(0) = \sigma^2 := \dint_{-\infty}^\infty R(r) dr 
   = 2\dint_0^\infty \E\{q(0)q(r)\}dr.
\end{equation}

Because \eqref{eq:u1eps1dhelm} is an average of random variables
decorrelating sufficiently fast, we expect a central limit-type result
to show that $u_{1\eps}(x,\omega)$ converges to a Gaussian random
variable. Combined with the variance \eqref{eq:hatR0}, we expect the
limit to be the following stochastic integral:
\begin{equation}
  \label{eq:u11dhelm}
  u_1(x,\omega) = -\sigma \dint_D G(x,y) u_0(y) dW_y(\omega),
\end{equation}
where $dW_y(\omega)$ is standard white noise on $({\cal C}(D),{\cal
  B}({\cal C}(D)),\P)$ \cite{billingsley1}. More precisely, we show the
following result:
\begin{theorem}
  \label{thm:cv1dhelm}
  Let us assume that $G(x,y)$ is Lipschitz continuous. Then, under the
  conditions of Proposition \ref{prop:strongcv}, the process
  $u_{1\eps}(x,\omega)$ converges weakly and in distribution in the
  space of continuous paths ${\cal C}(D)$ to the limit $u_1(x,\omega)$
  in \eqref{eq:u11dhelm}. As a consequence, the corrector to
  homogenization satisfies that
  \begin{equation}
    \label{eq:cvueps1dhelm}
    \dfrac{u_\eps-u_0}{\sqrt\eps}(x) \xrightarrow{\,\rm dist. \,}
    -\sigma \dint_D G(x,y) u_0(y) dW_y,\quad \mbox{ as }\eps\to0,
  \end{equation}
  in the space of integrable paths $L^1(D)$.
\end{theorem}
\begin{proof}
  We recall the classical result on the weak convergence of random
  variables with values in the space of continuous paths
  \cite{billingsley1}:
  \begin{proposition}
  \label{prop:2}
  Suppose $(Z_n; 1\leq n\leq \infty)$ are random variables with values
  in the space of continuous functions ${\cal C}(D)$. Then $Z_n$
  converges weakly (in distribution) to $Z_\infty$ provided that:
  \begin{itemize}
  \item[(a)] any finite-dimensional joint distribution
    $(Z_n(x_1),\ldots,Z_n(x_k))$ converges to the joint distribution
    $(Z_\infty(x_1),\ldots,Z_\infty(x_k))$ as $n\to\infty$.
  \item[(b)] $(Z_n)$ is a tight sequence of random variables. A
    sufficient condition for tightness of $(Z_n)$ is the following
    Kolmogorov criterion: there exist positive constants 
   $\nu$, $\beta$, and $\delta$ such that 
  \begin{equation}
    \label{eq:pretight}
    \begin{array}{ll}
       (i) & \sup\limits_{n \geq 1} \E\{|Z_n(t)|^\nu\} <\infty, \quad
       \mbox{ for some } t\in D, \\
       (ii) & \E\{|Z_n(s)-Z_n(t)|^\beta\} \lesssim |t-s|^{1+\delta},
    \end{array}
  \end{equation}
  uniformly in $n\geq1$ and $t,s\in D$.
  \end{itemize}
  \end{proposition}
  
  \paragraph{Tightness.}
  Tightness of $u_{1\eps}(x,\omega)$ is obtained with
  $\nu=\beta=2$ and $\delta=1$. Indeed, we easily obtain that
  \begin{displaymath}
  \E \{ |u_{1\eps}(x,\omega)|^2\} \lesssim 1,
  \end{displaymath}
  in fact uniformly in $x\in D$. Now by assumption on $G(x,y)$ we
  obtain that
  \begin{displaymath}
  \begin{array}{l}
  \E \{|u_{1\eps}(x,\omega)-u_{1\eps}(\xi,\omega)|^2\}
  = \E \Big( \dint_D [G(x,y)-G(\xi,y)]\dfrac1{\sqrt\eps}q(\frac y\eps)
    u_0(y)dy \Big)^2\\
   = \dint_{D^2} [G(x,y)-G(\xi,y)][G(x,\zeta)-G(\xi,\zeta)]
   \dfrac1\eps R(\dfrac{\zeta-y}{\eps}) u_0(y)u_0(\zeta) dy d\zeta\\
   \lesssim |x-\xi|^2 \dint_{D^2}
    \dfrac1\eps |R(\dfrac{\zeta-y}{\eps})| u_0(y)u_0(\zeta)dy d\zeta
   \lesssim |x-\xi|^2,
  \end{array}
  \end{displaymath}
  since the correlation function $R(r)$ is integrable and $u_0$ is
  bounded. This proves tightness of the sequence
  $u_{1\eps}(x,\omega)$, or equivalently weak convergence of the
  measures $\P_\eps$ generated by $u_{1\eps}(x,\omega)$ on $({\cal
    C}(D),{\cal B}({\cal C}(D)))$.
  
  \paragraph{Finite dimensional distributions.}
  Now any finite-dimensional distribution
  $(u_{1\eps}(x_j,\omega))_{1\leq j\leq n}$ has the characteristic
  function
  \begin{displaymath}
    \Phi_\eps(\bk) = \E \{ e^{i k_j u_{1\eps}(x_j,\omega)} \},\qquad
    \bk=(k_1,\ldots, k_n).
  \end{displaymath}
  The above characteristic function can be recast as
  \begin{displaymath}
    \Phi_\eps(\bk) = \E \{ e^{i \int_D m(y)\frac{1}{\sqrt\eps}
   q_\eps(y)dy }\},\qquad
    m(y) = \dsum_{j=1}^n k_j G(x_j,y) u_0(y).
  \end{displaymath}
  As a consequence, convergence of the finite dimensional
  distributions will be proved if we can show convergence of:
  \begin{equation}
    \label{eq:cvonemoment}
    I_{m\eps}:= \dint_D m(y) \dfrac{1}{\sqrt\eps} q(\dfrac y\eps)dy 
    \xrightarrow{\,\rm dist. \,}
   I_m :=\dint_D m(y) \sigma dW_y, \qquad \eps\to0,
  \end{equation}
  for arbitrary continuous moments $m(y)$. Such integrals have been
  extensively analyzed in the literature, see e.g.
  \cite{BP-SIAP-78,khasminskii}, where the above integral, for
  $D=(a,b)$ may be seen as the solution $x_\eps(b)$ of the following
  ordinary differential equation with random coefficients:
  \begin{displaymath}
    \dot x_\eps = \dfrac{1}{\sqrt\eps} q(\dfrac t\eps) m(t),\qquad
    x_\eps(a)=0.
  \end{displaymath}
  Since we will use the same methodology in higher space dimensions,
  we give a short proof of \eqref{eq:cvonemoment} using the central
  limit theorem for correlated discrete random variables as stated
  e.g. in \cite{B-AP-82}.

  \paragraph{Approximation by piecewise constant integrand.}
  Note that if we replace $m(y)$ by $m_h(y)$, then 
  \begin{equation}
    \label{eq:appmh}
    \E \{ (I_{m\eps}-I_{m_h\eps})^2\} \lesssim \|m-m_h\|_\infty^2,
  \end{equation}
  where $\|\cdot\|_\infty$ is the uniform norm on $D$. It is therefore
  sufficient to consider \eqref{eq:cvonemoment} for a sequence of
  functions $m_h$ converging to $m$ in the uniform sense. Since $m$ is
  (uniformly) continuous, we can approximate it by piecewise constant
  functions $m_h$ that are constant on $M$ intervals of size
  $h=\frac{b-a}M$. Let $m_{hj}$ be the value of $m_h$ on the $j^{\rm
    th}$ interval and define the random variables
  \begin{displaymath}
    M_{\eps j} = m_{hj} \dint_{(j-1)h}^{jh}
     \dfrac{1}{\sqrt\eps} q(\dfrac y\eps)dy .
  \end{displaymath}
  \paragraph{Independence of random variables.}
  We want to show that the variables $M_{\eps j}$ become independent
  in the limit $\eps\to0$. This is done by showing that
  \begin{displaymath}
    {\cal E}(\bk)=\Big|\E\{e^{i\sum_{j=1}^M k_jM_{\eps j}}\} - \prod_{j=1}^M 
        \E\{e^{ik_jM_{\eps j}}\}\Big| \to0 \,\, \mbox{ as } \eps\to0, \quad
  \end{displaymath}
  for all $\bk=\{k_j\}_{1\leq j\leq M} \in \Rm^{M}$.  Let
  $\bk\in\Rm^M$ fixed, $0<\eta<\frac h2$ and define
  \begin{displaymath}
    P^\eta_{\eps j} =  m_{hj} \dint_{(j-1)h+\eta}^{jh-\eta}
     \dfrac{1}{\sqrt\eps} q(\dfrac y\eps)dy ,\qquad 
    Q^\eta_{\eps j} = M_{\eps j} - P^\eta_{\eps j}.
  \end{displaymath}
  Now we write 
  \begin{displaymath}
    \E\{e^{i\sum_{j=1}^M k_jM_{\eps j}}\} =
    \E\{[e^{i k_1 Q^\eta_{\eps 1}}-1]e^{ik_1 P^\eta_{\eps 1}
         +i\sum_{j=2}^M k_jM_{\eps j}}\}  +
     \E\{e^{ik_1 P^\eta_{\eps 1}+i\sum_{j=2}^M k_jM_{\eps j}}\}.
  \end{displaymath}
  Using the strong mixing condition \eqref{eq:strongmixing}, we find
  that
  \begin{displaymath}
    \Big|\E\{e^{ik_1 P^\eta_{\eps 1}+i\sum_{j=2}^M k_jM_{\eps j}}\}
    - \E\{e^{ik_1 P^\eta_{\eps 1}}\}\E\{e^{i\sum_{j=2}^M k_jM_{\eps j}}\}\Big|
   \lesssim \varphi(\frac{2\eta}{\eps}).
  \end{displaymath}
  Now we find that $\E\{Q^\eta_{\eps j}\}=0$ and $\E\{[Q^\eta_{\eps
    j}]^2\}\lesssim\eta$. The latter result comes from integrating
  $\eps^{-1}R(\frac{t-s}\eps)ds dt$ over a cube of size $O(\eta^2)$.
  Since $|e^{ix}-1|\lesssim |x|$, we deduce that
  \begin{displaymath}
    |\E\{[e^{i k_1 Q^\eta_{\eps 1}}-1]e^{ik_1 P^\eta_{\eps 1}
         +iZ}\}|
    \leq \E\{[e^{i k_1 Q^\eta_{\eps 1}}-1]^2\}^{\frac12}
   \lesssim \eta^{\frac12},
  \end{displaymath}
  for an arbitrary random variable $Z$ (equal to $0$ or to
  $\sum_{j=2}^M k_jM_{\eps j}$ here). Thus,
  \begin{displaymath}
     \Big|\E\{e^{ik_1 M_{\eps 1}+i\sum_{j=2}^M k_jM_{\eps j}}\}
    - \E\{e^{ik_1 M_{\eps 1}}\}\E\{e^{i\sum_{j=2}^M k_jM_{\eps j}}\}\Big|
   \lesssim \varphi(\frac{2\eta}{\eps}) + \eta^{\frac12}.
  \end{displaymath}
  By induction, we thus find that for all $0<\eta<\frac h2$, 
  \begin{displaymath}
    {\cal E} \lesssim M\varphi(\frac{2\eta}{\eps}) + \eta^{\frac12}.
  \end{displaymath}
  This expression tends to $0$ say for $\eta=\eps^{\frac12}$. This
  shows that the random variables $M_{\eps j}$ become independent as
  $\eps\to0$. We show below that each $M_{\eps j}$ converges to a
  centered Gaussian variable as $\eps\to0$. The sum over $j$ thus
  yields in the limit a centered Gaussian variable with variance the
  sum of the $M$ individual variances.

  \paragraph{Central Limit Theorem for discrete random variables.}
  By stationarity of the process $q(x,\omega)$, we are thus led to
  showing that
  \begin{displaymath}
    \dint_0^h  \dfrac{1}{\sqrt\eps} q(\dfrac y\eps)dy  
     \xrightarrow{\,\rm dist. \,}\dint_0^h \sigma dW_y
      = \sigma W_h = \sigma {\cal N}(0,h), \qquad \eps\to0,
  \end{displaymath}
  where ${\cal N}(0,h)$ is the centered Gaussian variable with
  variance $h$.  We break up $h$ into $N=h/\eps$ (which we assume is
  an integer) intervals and call
  \begin{displaymath}
    q_j = \dint_{(j-1)\eps}^{j\eps} \dfrac{1}{\eps} q(\dfrac y\eps)dy
    = \dint_{j-1}^j q(y)dy,
   \qquad j\in\Zm.
  \end{displaymath}
  The $q_j$ are stationary mixing random variables and we are interested
  in the limit 
  \begin{equation}
    \label{eq:discsum1d}
    \sqrt\eps\sum_{j=1}^N q_j= \dfrac{\sqrt h}{\sqrt N}
    \sum_{j=1}^N q_j.
  \end{equation}
  Following remark 3 in \cite{B-AP-82}, we introduce ${\cal A}_m$ and
  ${\cal A}^m$ as the $\sigma-$algebras generated by $(q_j)_{j\leq m}$
  and $(q_j)_{j\geq m}$, respectively. Let then 
  \begin{equation}
    \label{eq:discreterho1d}
    \rho(n) = \sup\Big\{ \dfrac{\E\big\{(\eta-\E\{\eta\})(\xi-\E\{\xi\})\big\}}
     {\big(\E\{\eta^2\}\E\{\xi^2\}\big)^{\frac12}}; \,\eta\in 
     L^2({\cal A}_0),\quad \xi\in L^2({\cal A}^n\} \Big\}.
  \end{equation}
  Then provided that $\sum_{n\geq1} \rho(n) <\infty$, we obtain the
  following central limit theorem
  \begin{equation}
    \label{eq:clt}
    \dfrac{\sqrt h}{\sqrt N} \sum_{j=1}^N q_j \xrightarrow{\,\rm dist. \,}
    \sqrt h \sigma {\cal N}(0,1) \equiv \sigma {\cal N}(0,h) ,
  \end{equation}
  where ${\cal N}(0,1)$ is the standard normal variable, where
  $\equiv$ is used to mean equality in distribution, and where
  $\sigma^2=\sum_{n\in\Zm} \E\{q_0 q_n\}$. It remains to verify that the two
  definitions of $\sigma$ above and in \eqref{eq:hatR0} agree and that
  $\sum_{n\geq1} \rho(n) <\infty$. Note that 
  \begin{displaymath}
    \sum_{n\in\Zm} \E\{q_0 q_n\} = \!
    \dint_0^1\!\! \dint_{-\infty}^\infty\!\!\!\!
      \E\{q(y)q(z)\} dydz = \!\dint_0^1\!\! \dint_{-\infty}^\infty\!\!\!\!
     \E \{ q(y)q(y+z)\} dydz = \!\dint_0^1\!\!\hat R(0)dy = \hat R(0),
  \end{displaymath}
  thanks to \eqref{eq:hatR0}. Now we observe that
  $\rho(n)\leq\varphi(n-1)$ so that summability of $\rho(n)$ is
  implied by the integrability of $\varphi(r)$ on $\Rm^+$. This
  concludes the proof of the convergence in distribution of
  $u_{1\eps}$ in the space of continuous paths ${\cal C}(D)$.
 
  It now remains to recall the convergence result
  \eqref{eq:limuuepscorr} to obtain \eqref{eq:cvueps1dhelm} in the
  space of integrable paths.
\end{proof}

\subsection{Oscillatory integral in arbitrary space dimensions}
\label{sec:helmmd}

In dimension $1\leq d\leq 3$, the leading term in the corrector
$\eps^{-\frac d2}(u_\eps-u_0)$ is given by:
\begin{equation}
  \label{eq:u1epsmdhelm}
  u_{1\eps}(\bx,\omega)=
  \dint_D -G(\bx,\by)  \dfrac{1}{\eps^{\frac d2}}
    q_\eps(\by,\omega) u_0(\by) d\by.
\end{equation}
The variance of $u_{1\eps}(\bx,\omega)$ is given by
\begin{displaymath}
  \E\{u_{1\eps}^2(\bx,\omega)\} =
  \dint_{D^2} G(\bx,\by)G(\bx,\bz) \dfrac{1}{\eps^d} 
  R\Big(\dfrac{\by-\bz}\eps\Big) u_0(\by)u_0(\bz) d\by d\bz.
\end{displaymath}
As in the one-dimensional case, it converges as $\eps\to0$ to the limit
\begin{equation}
  \label{eq:limitu1md}
  \E\{u_{1}^2(\bx,\omega)\} 
   =\sigma^2 \dint_D G^2(\bx,\by) u_0^2(\by)d\by,\qquad
   \sigma^2 = \dint_{\Rm^d} \E \{ q(\bzero)q(\by)\} d\by. 
\end{equation}
Because of the singularities of the Green's function $G(\bx,\by)$ in
dimension $d\geq2$, we prove here less accurate results than those
obtained in dimension $d=1$ in the preceding section.

We want to obtain convergence of the above corrector in distribution
on $(\Omega,{\cal F},\P)$ and weakly in $D$. More precisely, let
$M_k(\bx)$ for $1\leq k\leq K$ be sufficiently smooth functions such
that
\begin{equation}
  \label{eq:mk}
  m_k(\by) = -\dint_D M_k(\bx)G(\bx,\by) u_0(\by)  d\bx
    = -{\cal G}M_k(\by) u_0(\by) ,\quad 1\leq k\leq K,
\end{equation}
are continuous functions (we thus assume that $u_0(\bx)$ is continuous
as well). Let us introduce the random variables
\begin{equation}
  \label{eq:Ikeps}
  I_{k\eps}(\omega) = \dint_D m_k(\by) \dfrac{1}{\eps^{\frac d2}}
    q\Big(\dfrac{\by}{\eps},\omega\Big) d\by.
\end{equation}
Because of hypothesis [H3], the accumulation points of the integrals
$I_{k\eps}(\omega)$ are not modified if $q(\frac{\by}{\eps},\omega)$
is replaced by $q_\eps(\by,\omega)$. The main result of this section
is the following:
\begin{theorem}
  \label{thm:oscintmd}
  Under the above conditions and the hypotheses of Proposition
  \ref{prop:strongcv}, the random variables $I_{k\eps}(\omega)$
  converge in distribution to the mean zero Gaussian random variables
  $I_k(\omega)$ as $\eps\to0$, where the correlation matrix is given
  by
  \begin{equation}
    \label{eq:Sigma}
    \Sigma_{jk}=\E\{I_jI_k\} = \sigma^2\dint_D m_j(\by) m_k(\by)
       d\by,
  \end{equation}
  where $\sigma$ is given by
  \begin{equation}
    \label{eq:sigmamd}
    \sigma^2 = \dint_{\Rm^d} \E \{ q(\bzero)q(\by)\} d\by.
  \end{equation}
  Moreover, we have the stochastic representation
  \begin{equation}
    \label{eq:Ikrand}
    I_k(\omega) = \dint_D m_k(\by) \sigma dW_{\by},
  \end{equation}
  where $dW_{\by}$ is standard multi-parameter Wiener process
  \cite{Khos-SV-02}.

  As a result, for $M(\bx)$ sufficiently smooth, we obtain that
  \begin{equation}
    \label{eq:convmtsmd}
    \Big(\dfrac{u_\eps - u_0}{\eps^{\frac d2}}, M \Big)
    \xrightarrow{\,\rm dist. \,}
    -\sigma  \dint_D {\cal G}M(\by){\cal G}f(\by) dW_{\by}.
  \end{equation}
\end{theorem}
\begin{proof}
  The convergence in \eqref{eq:convmtsmd} is a direct consequence of
  \eqref{eq:Ikrand} since 
  \begin{displaymath}
    \dint_{D^2} M(\bx) G(\bx,\by) u_0(\by) dW_{\by} d\bx=
    \dint_D {\cal G}M(\by){\cal G}f(\by) dW_{\by},
  \end{displaymath}
  and of the strong convergence \eqref{eq:limuuepscorr} in Proposition
  \ref{prop:strongcv}.  The equality \eqref{eq:Ikrand} is directly
  deduced from \eqref{eq:Sigma} since $I_k(\omega)$ is a
  (multivariate) Gaussian variable.  In order to prove
  \eqref{eq:Sigma}, we use a methodology similar to that in the proof
  of Theorem \ref{thm:cv1dhelm}.
  
  The characteristic function of the random variables
  $I_{k\eps}(\omega)$ is given by
  \begin{displaymath}
    \Phi_\eps(\bk) = \E \{ e^{i \sum_{k=1}^Kk_j I_{j\eps}(\omega)} \},\qquad
    \bk=(k_1,\ldots, k_K),
  \end{displaymath}
  and may be recast as
  \begin{displaymath}
    \Phi_\eps(\bk) = \E \{ e^{i \int_D m(y) \eps^{\frac {-d}2}
    q(\frac{\by}{\eps},\omega)d\by}\}, \quad
    m(\by) = \dsum_{j=1}^K k_j m_j(\by).
  \end{displaymath}
  So \eqref{eq:Sigma} follows from showing that 
  \begin{equation}
    \label{eq:cvIepsmd}
    I_{\eps}(\omega) = \dint_D m(\by) \dfrac{1}{\eps^{\frac d2}}
    q\Big(\dfrac{\by}{\eps},\omega\Big) d\by
   \xrightarrow{\,\rm dist. \,} \dint_D m(\by) \sigma dW_{\by},
  \end{equation}
  for an arbitrary continuous function $m(\by)$. As in the
  one-dimensional case and for the same reasons, we replace $m(\by)$
  by $m_h(\by)$, which is constant on small hyper-cubes ${\cal C}_j$
  of size $h$ (and volume $h^d$) and that there are $M\approx h^{-d}$
  of them. Because $\partial D$ is assumed to be sufficiently smooth,
  it can be covered by $M_S\approx h^{-d+1}$ cubes and we set
  $m_h(\bx)=0$ on those cubes. The contribution to $I_\eps(\omega)$ is
  seen to converge to $0$ as $h\to0$ in the mean-square sense as in
  \eqref{eq:appmh}.
  
  We define the random variables
  \begin{displaymath}
    M_{\eps j}(\omega) = m_{hj} \dint_{{\cal C}_j}
     \dfrac{1}{\eps^{\frac d2}} q(\dfrac \by\eps,\omega)d\by,\qquad
    1\leq j\leq M,
  \end{displaymath}
  where $m_{hj}$ is the value of $m_h$ on ${\cal C}_j$ and are
  interested in the limiting distribution as $\eps\to0$ of the random
  variable
  \begin{equation}
    \label{eq:Iepsh}
    I_\eps^h (\omega) = \dsum_{j=1}^M M_{\eps j}(\omega).
  \end{equation}
  We show below that these random variables are again independent in
  the limit $\eps\to0$ and each variable converges to a centered
  Gaussian variable. As a consequence, $I_\eps^h(\omega)$ converges
  in distribution to a centered Gaussian variable whose variance is
  the sum of the variances of the variables $M_{\eps j}(\omega)$ in
  the limit $\eps\to0$.
  
  That the random variables $M_{\eps j}$ are independent in the limit
  $\eps\to0$ is shown using a similar method to that of the
  one-dimensional case. We want to obtain that
  \begin{displaymath}
    {\cal E}(\bk)=\Big|\E\{e^{i\sum_{j=1}^M k_jM_{\eps j}}\} - \prod_{j=1}^M 
        \E\{e^{ik_jM_{\eps j}}\}\Big| \to0 \quad \mbox{ as } \eps\to0, \quad
    \mbox{ for all } \bk=\{k_j\}_j \in \Rm^{M}.
  \end{displaymath}
  Let $0<\eta<\frac h2$ and ${\cal D}_j^\eta=\{\bx\in{\cal C}_j;
  d(\bx,\partial{\cal C}_j)>\eta\}$. We define
  \begin{displaymath}
    P^\eta_{\eps j} =  m_{hj} \dint_{{\cal D}_j^\eta} 
    \dfrac{1}{\eps^{\frac d2}} q(\dfrac \by\eps,\omega)d\by,\qquad
    Q^\eta_{\eps j} = M_{\eps j} - P^\eta_{\eps j}.
  \end{displaymath}
  We write again:
  \begin{displaymath}
    \E\{e^{i\sum_{j=1}^M k_jM_{\eps j}}\} =
    \E\{[e^{i k_1 Q^\eta_{\eps 1}}-1]e^{ik_1 P^\eta_{\eps 1}
         +i\sum_{j=2}^M k_jM_{\eps j}}\}  +
     \E\{e^{ik_1 P^\eta_{\eps 1}+i\sum_{j=2}^M k_jM_{\eps j}}\}.
  \end{displaymath}
  Using the strong mixing condition \eqref{eq:strongmixing}, we find
  that
  \begin{displaymath}
    \Big|\E\{e^{ik_1 P^\eta_{\eps 1}+i\sum_{j=2}^M k_jM_{\eps j}}\}
    - \E\{e^{ik_1 P^\eta_{\eps 1}}\}\E\{e^{i\sum_{j=2}^M k_jM_{\eps j}}\}\Big|
   \lesssim \varphi(\frac{2\eta}{\eps}).
  \end{displaymath}
  We find as in the one-dimensional case that $\E\{Q^\eta_{\eps
    j}\}=0$ and $\E\{[Q^\eta_{\eps j}]^2\}\lesssim\eta
  h^{(d-1)}\lesssim \eta$ with a bound independent of $\eps$. This
  comes from integrating $\eps^{-d}R(\frac{\bx-\by}\eps)d\bx d\by$ on
  a domain of size $O([\eta h^{d-1}]^2)$.  The rest of the proof
  follows as in the one-dimensional case.
  
  It remains to address the convergence of $M_{\eps j}$ as $\eps\to0$.
  By invariance of $q(\bx)$, it is sufficient to consider integrals on
  the cube $[\bzero,\bh]$, with $\bh=(h,\ldots,h)$. It now remains to
  show that
  \begin{equation}
    \label{eq:convcube}
    \dint_{[\bzero,\bh]} \dfrac{1}{\eps^{\frac d2}}
    q\Big(\dfrac{\by}{\eps},\omega\Big)d\by  \xrightarrow{\,\rm dist. \,}
    \sigma \dint_{[\bzero,\bh]} dW_{\by} = \sigma {\cal N}(0,h^d).
  \end{equation}
  For a multi-index $\bj\in\Zm^d$, we define 
  \begin{displaymath}
    q_{\bj}(\omega) = \dint_{\bj+[\bzero,\bone]} q(\by,\omega) d\by.
  \end{displaymath}
  Then \eqref{eq:convcube} will follow by homogeneity if we can show
  that 
  \begin{equation}
    \label{eq:sumqsmd}
    \dfrac{1}{\sigma n^{\frac d2}} \dsum_{\bj\in[\bzero,\bn]}
    q_\bj \xrightarrow{\,\rm dist. \,} {\cal N}(0,1).
  \end{equation}
  The latter result is proved in e.g. \cite{B-AP-82,Doukhan95}.  The
  results in these references are stated in terms of $\alpha$-mixing
  coefficients. Since $\alpha$ coefficients are bounded by $\rho$
  coefficients \cite[p.4]{Doukhan95}, we state the results in terms of
  less optimal $\rho$-mixing coefficients.
  
  Let $A$ and $B$ be subsets of $\Zm^d$ and let ${\cal A}$ and ${\cal
    B}$ be the $\sigma$ algebras generated by $q_{\bj}$ on $A$ and
  $B$, respectively. Then we define 
  \begin{displaymath}
    \rho(n)= \sup\Big\{ \dfrac{\E\big\{(\eta-\E\{\eta\})(\xi-\E\{\xi\})\big\}}
     {\big(\E\{\eta^2\}\E\{\xi^2\}\big)^{\frac12}}; \,\eta\in 
     L^2({\cal A}),\quad \xi\in L^2({\cal B}\},\quad
      d(A,B)\geq n \Big\}.
  \end{displaymath}
  We then assume that $\E\{q_{\bj}^6\}<\infty$ as in hypothesis [H2] and 
  that $\rho(n)=o(n^{-d})$ and that 
  \begin{equation}
    \label{eq:decayrhond}
    \dsum_{n=0}^\infty n^{d-1} \rho^{\frac12}(n) <\infty.
  \end{equation}
  Then we verify that the hypotheses in \cite{B-AP-82} (see also
  \cite[p.48]{Doukhan95}) are satisfied so that \eqref{eq:sumqsmd}
  holds with
  \begin{displaymath}
    \sigma^2 = \dsum_{\bj\in\Zm^d}\E\{q_{\bzero}q_{\bj}\}.
  \end{displaymath}
  We verify as in the one-dimensional case that the above $\sigma$
  agrees with that in definition \eqref{eq:sigmamd}. Now we verify
  that \eqref{eq:decayrhond} is a consequence of the integrability of
  $r^{d-1}\varphi^{\frac12}(r)$. The decay $\rho(n)=o(n^{-d})$ is
  obtained when $\varphi(r)$ decays faster than $r^{-d-\eta}$ for some
  $\eta>0$; see Remark \ref{rem:varphi}.
\end{proof}

\subsection{Larger fluctuations, random and periodic homogenization}
\label{sec:larger}

We now consider several generalizations of the results presented in
earlier sections and compare homogenization in periodic and random
media. The results stated in the preceding sections, corresponding to
the case $\alpha=0$ below, generalize to larger fluctuations of the
form:
\begin{equation}
  \label{eq:tildeqepslarge}
  \tilde q_\eps(\bx,\omega) = \dfrac{1}{\eps^{\alpha d}} 
    q\Big(\dfrac{\bx}\eps,\omega\Big).
\end{equation}
The corrector $-{\cal G} q_\eps {\cal G} f$ is now of order
$\eps^{d(\frac12-\alpha)}$ for $0\leq\alpha<\frac12$. The next-order
corrector, given by ${\cal G}q_\eps {\cal G}q_\eps {\cal G} f$ in
\eqref{eq:expansioncorr}, is bounded in $L^1(\Omega\times D)$ by
$\eps^{d(\frac{1+\eta}{2+\eta}-2\alpha)}$ according to Lemma
\ref{lem:gqgqgf}. The order of this term is smaller than the order of
the leading corrector $\eps^{d(\frac12-\alpha)}$ again provided that
$0\leq\alpha<\frac{\eta}{2(2+\eta)}$, which converges to $\frac12$ for 
$d=1,2$ as $\eta\to\infty$ and converges to $\frac16$ for $d=3$ as
$\eta\to1$.

In dimensions $d=1,2$, we can infer from these results that
$\eps^{-d(\frac12-\alpha)}(u_\eps-u_0)$ converges in distribution to
the limits obtained in the preceding sections as $\eps\to0$ provided
that $0\leq\alpha<\frac12$. The proof presented in this paper extends
to the values $0\leq\alpha<\frac14$. Indeed, the proof is based on
imposing that the spectral radius of ${\cal G}q_\eps{\cal G}q_\eps$ is
sufficiently small using \eqref{eq:bdgqgq} in Lemma
\ref{lem:boundgqgq}, which for \eqref{eq:tildeqepslarge}, translates
into $\E\{ \|{\cal G}q_\eps{\cal G}q_\eps \|^2_{{\cal L}(L^2(D))}\}
\lesssim \eps^{d(1-4\alpha)}$. We then verify that all results leading
to Proposition \ref{prop:strongcv} generalize when $0<\alpha<\frac14$
to yield \eqref{eq:limuuepscorr} with $\eps^{\frac d2}$ replaced by
$\eps^{d(\frac12-\alpha)}$. A proof of convergence for
$0\leq\alpha<\frac12$ would presumably require us to analyze all the
terms in the formal expansion
\begin{equation}
  \label{eq:formalexp}
  u_\eps=\dsum_{k=0}^\infty(-{\cal
  G}q_\eps)^k{\cal G}f,
\end{equation}
something we do not address here.  In the limiting case
$\alpha=\frac12$, the above theory breaks down and $u_\eps$ no longer
converges to the deterministic solution $u_0$ as is shown in the
temporal one-dimensional case in \cite{PP-GAK-06}.

\medskip

The results on the corrector $u_\eps-u_0$ obtained in the preceding
sections, namely Theorems \ref{thm:cv1dhelm} and \ref{thm:oscintmd}
are valid for $1\leq d\leq 3$. If we admit that the expansion
\eqref{eq:formalexp} involves a first term $u_0$, a second term
$-{\cal G}q_\eps u_0$, and smaller order terms, then the results
obtained in Theorem \ref{thm:oscintmd} show that $u_\eps-u_0$
converges weakly in space and in distribution to a process of order
$O(\eps^{\frac d2})$.  The critical case $d=4$ yields a correction of
order $\eps^2$, whereas $\eps^{\frac d2}$ would be even smaller for
$d\geq5$. 

The theory presented in this paper does not allow us to justify
\eqref{eq:formalexp} when $d\geq4$ because the corresponding Green's
function are no longer square integrable.  Another argument shows that
corrections of order $\eps^2$ correspond to a transition and that we
should not expect quite the same results for $d\leq3$ and $d\geq4$.
Indeed, let us consider the problem in the periodic case:
\begin{equation}
  \label{eq:per}
  \begin{array}{ll}
     -\Delta u_\eps + q\Big(\dfrac{\bx}\eps\Big) u_\eps =f & D\\
     u_\eps=0 & \partial D,
  \end{array}
\end{equation}
on a smooth open, bounded, domain $D\subset\Rm^d$, where $q(\by)$ is
$[0,1]^d$-periodic. Then following \cite{blp1}, we introduce the fast
scale $\by=\frac{\bx}\eps$ and introduce a function
$u_\eps=u_\eps(\bx,\frac{\bx}\eps)$. Gradients $\nabla_\bx$ become
$\frac1\eps\nabla_\by+\nabla_\bx$ and \eqref{eq:per} becomes formally
\begin{displaymath}
  \Big(-\dfrac{1}{\eps^2}\Delta_\by - \dfrac2\eps\nabla_\bx\cdot\nabla_\by
   -\Delta_\bx + q(\by)\Big) u_\eps(\bx,\by) = f(\bx).
\end{displaymath}
Plugging the expansion $u_\eps=u_0+\eps u_1+\eps^2 u_2$ into the above
equality and equating like powers of $\eps$ yields three equations.
The first equation shows that $u_0=u_0(\bx)$. The second equation
shows that $u_1=u_1(\bx)$, which we can choose as $u_1\equiv0$. The
third equation
\begin{math}
  -\Delta_\by u_2 - \Delta_\bx u_0 +q(\by) u_0 = f(\bx),
\end{math}
admits a solution provided that 
\begin{displaymath}
  - \Delta_\bx u_0 + \langle q \rangle u_0 = f(\bx),\quad D
\end{displaymath}
with $u_0=0$ on $\partial D$. Here, $\langle q \rangle$ is the average
of $q$ on $[0,1]^d$, which we assume is sufficiently large that the
above equation admits a unique solution.  We recast the above equation
as $u_0={\cal G}_D f$.  The corrector $u_2$ thus solves
\begin{displaymath}
  -\Delta_\by u_2  = \Big( \langle q \rangle - q(\by)\Big) u_0(\bx),
\end{displaymath}
and is uniquely defined along with the constraint $\langle u_2
\rangle=0$.  We denote the solution operator of the above cell problem
as ${\cal G}_{\#}$ so that $u_2=-{\cal G}_{\#}(q-\langle q \rangle)
{\cal G}f$. Thus formally, we have obtained that
\begin{equation}
  \label{eq:corrper}
  u_\eps(\bx) = {\cal G}f(\bx) - \eps^2 {\cal G}_{\#}(q-\langle q \rangle)
    \Big(\frac{\bx}\eps\Big){\cal G}f(\bx)  + \mbox{ l.o.t. }
\end{equation}
We thus observe that the corrector
$u_{2\eps}(\bx):=u_2(\bx,\frac{\bx}{\eps})$ is of order $O(\eps^2)$ in
the $L^2$ sense, say. In the sense of distributions, however, the
corrector may be of order $o(\eps^m)$ for all integer $m$ in the sense
that $\int_D M(\bx) u_{2\eps}(\bx)d\bx\ll\eps^m$ for all $m$ when
$M(\bx)u_0(\bx)\in {\cal C}^\infty_0(D)$.

A similar behavior occurs for the random corrector
\begin{equation}
  \label{eq:v1eps}
  v_{1\eps}(\bx,\omega) =\dint_D -G(\bx,\by) q\Big(\dfrac{\by}\eps,\omega\Big)
  u_0(\by) d\by.
\end{equation}
Theorem \ref{thm:oscintmd} shows that $(v_{1\eps},M(\bx))$ is of order
$O(\eps^{\frac d2})$ for $M(\bx)$ and $u_0(\bx)$ sufficiently smooth
and that $\eps^{-\frac d2}(v_{1\eps},M(\bx))$ converges in
distribution to a Gaussian random variable. This result, however, does
not hold in the $L^2(D)-$sense for $d\geq4$ when $G(\bx,\by)$ is the
fundamental solution of the Helmholtz equation $-\Delta+q_0(\bx)$ on
$D$.  Indeed, we can prove that
\begin{proposition}
  \label{prop:l2norm}
  Provided that $u_0(\bx)$ and $\hat R(\bxi)$ are sufficiently smooth,
  we obtain that:
  \begin{equation}\label{eq:convv1}
  \E\{v_{1\eps}^2(\bx,\omega)\} \sim \left\{
    \begin{array}{ll}
      \eps^d \hat R(\bzero) \dint_{D} G^2(\bx,\by) u_0^2(\by)d\by 
                \quad& 1\leq d\leq 3\\
      \eps^4 |\ln\eps|  
    \dfrac{(2\pi)^4\hat R(\bzero)}{c_4} u_0^2(\bx) & d=4 \\
      \eps^4 u_0^2(\bx) (2\pi)^d \dint_{\Rm^d}
       \dfrac{\hat R(\bxi)}{|\bxi|^4} d\bxi & d\geq5.
    \end{array}\right.
   \end{equation}
   Here $a_\eps\sim b_\eps$ means $a_\eps=b_\eps(1+o(1))$.
\end{proposition}
\begin{proof}
  We calculate:
  \begin{equation}
    \label{eq:v1epsint}
    \E\{v_{1\eps}^2(\bx,\omega)\} = \dint_{D^2} G(\bx,\by)G(\bx,\bz)
    R\Big(\dfrac{\by-\bz}\eps\Big) u_0(\by) u_0(\bz) d\by d\bz.
  \end{equation}
  Extending $G(\bx,\cdot)$ by $0$ outside of $D$, by the Parseval
  equality this is equal to
  \begin{displaymath}
    (2\pi)^d\dint_{\Rm^{2d}} |{\cal F}_{\by\to\bxi} 
     (G(\bx,\by)u_0(\by))|^2(\bxi)
    \eps^d \hat R(\eps\bxi) d\bxi,
  \end{displaymath}
  where ${\cal F}_{\bx\to\bxi}$ is the Fourier transform from $\bx$ to
  $\bxi$.  In dimension $1\leq d\leq 3$, since $\hat R(\eps\bxi)\to
  \hat R(\bzero)$ pointwise, the Lebesgue dominated convergence
  theorem yields the result. In dimension $d\geq4$, however, the Green
  function is no longer integrable and the integral is larger than
  $\eps^d$.
  
  Let us consider the cases $d\geq4$. We first replace $G(\bx,\by)$ by
  $c_d|\bx-\by|^{2-d}$ where $c_d$ is the measure of the unit sphere
  $S^{d-1}$. The difference $G(\bx,\by)-c_d|\bx-\by|^{2-d}$ is a
  function bounded by $C|\bx-\by|^{3-d}$, which yields a smaller
  contribution to $\E\{v_{1\eps}^2\}$. We leave the details to the
  reader. We also replace $u_0(\by)$ by $u_0(\bx)$, up to an error
  bounded by $|\bx-\by|^\alpha$ as soon as $u_0(\bx)$ is of class
  ${\cal C}^{0,\alpha}(D)$. This contribution again provides a lower
  order term to $\E\{v_{1\eps}^2\}$. Similarly, we replace 
  $u_0(\bz)$ by $u_0(\bx)$ and thus obtain that 
  \begin{displaymath}
    \E\{v_{1\eps}^2(\bx,\omega)\} \sim u_0^2(\bx)
    \dint_{D^2} \dfrac{1}{c_d|\bx-\by|^{d-2}}
     \dfrac{1}{c_d|\bx-\bz|^{d-2}}R\Big(\dfrac{\by-\bz}\eps\Big) d\by d\bz.
  \end{displaymath}
  Let $\alpha>0$ and $B(\bx,\alpha)$ the ball of center $\bx$ and
  radius $\alpha$ so that $B(\bx,\alpha)\subset D$. Because all
  singularities occur when $\by$ and $\bz$ are in the vicinity of
  $\bx$, we use the proof of the case $1\leq d\leq 3$ to show that up
  to a term of order $\eps^{d}$, we can replace $D$ by $B(\bx,\alpha)$
  so that
  \begin{equation}  \label{eq:v1interm}
    \E\{v_{1\eps}^2(\bx,\omega)\} \sim u_0^2(\bx)
    \dint_{B^2(\bzero,\alpha)} \dfrac{1}{c_d|\by|^{d-2}}
     \dfrac{1}{c_d|\bz|^{d-2}}R\Big(\dfrac{\by-\bz}\eps\Big) d\by d\bz.
  \end{equation}
  Now for $d\geq 5$, using the dominated convergence theorem, we can
  replace $B(\bzero,\alpha)$ by $\Rm^d$ because the Green function is
  square integrable at infinity, whence
  \begin{displaymath}
     \E\{v_{1\eps}^2(\bx,\omega)\} \sim u_0^2(\bx)
    \dint_{\Rm^{2d}} \dfrac{1}{c_d|\by|^{d-2}}
     \dfrac{1}{c_d|\bz|^{d-2}}R\Big(\dfrac{\by-\bz}\eps\Big) d\by d\bz.
  \end{displaymath}
  This, however, by the Parseval equality, is equal to
  \begin{displaymath}
    \E\{v_{1\eps}^2(\bx,\omega)\} \sim u_0^2(\bx) (2\pi)^d
    \dint_{\Rm^{d}} \dfrac{1}{|\bxi|^4} \eps^d \hat R(\eps\bxi) d\bxi
    = u_0^2(\bx) (2\pi)^d
    \dint_{\Rm^{d}} \dfrac{1}{|\bxi|^4} \eps^4 \hat R(\bxi) d\bxi,
  \end{displaymath}
  since the Fourier transform of the fundamental solution of the
  Laplacian is $|\bxi|^{-2}$.
  
  When $d=4$, we come back to \eqref{eq:v1interm}, and replace one of
  the integrals on $B(\bzero,\alpha)$ by an integral on $\Rm^d$ using
  again the dominated convergence theorem. This yields the term
  \begin{displaymath}
    \begin{array}{l}
    \dint_{B(\bzero,\alpha)\times\Rm^d}
    \dfrac{1}{c_4^2|\by|^2|\bz|^2} R\Big(\dfrac{\by-\bz}\eps\Big) d\by d\bz
   = \dint_{B(\bzero,\alpha)\times\Rm^d}
     \dfrac{(2\pi)^4\eps^2}{c_4|\by|^2|\bxi|^2} \hat R(\bxi) 
     e^{i\frac{\bxi\cdot\by}\eps} d\bxi d\by\\
    =\dint_{B(\bzero,\frac{\alpha}\eps)\times\Rm^d}
     \dfrac{(2\pi\eps)^4}{c_4 |\by|^2|\bxi|^2} \hat R(\bxi)
     e^{i\bxi\cdot\by} d\bxi d\by = 
    \hat R(\bzero) (2\pi\eps)^4 \dint_{B(\bzero,\frac{\alpha}\eps)}
    \dfrac{1}{c_4^2|\by|^4} d\by +O(\eps^4) \\
    = \dfrac{\hat R(\bzero) (2\pi\eps)^4}{c_4}
     \dint_0^{\frac\alpha\eps} \dfrac{|\by|^3}{|\by|^4} d|\by| +O(\eps^4)
    =  \dfrac{\hat R(\bzero) (2\pi\eps)^4}{c_4} |\ln\eps| +O(\eps^4).
    \end{array}
  \end{displaymath}
  Here, we have assumed that $|\hat R(\bxi)-\hat R(\bzero)|$ was
  bounded by $C|\bxi|^\beta$ for some $\beta>0$.
\end{proof}

In all dimensions, we thus obtain that $\eps^{-\frac
  d2}v_{1\eps}(\bx,\omega)$ converges (weakly and in distribution) to
a limit $u_1(\bx,\omega)=-\int_D G(\bx,\by)u_0(\by) dW_\by$. In
dimensions $1\leq d\leq 3$, we have proved that $u_1$ was the limit of
the corrector to homogenization $\eps^{-\frac d2}(u_\eps-u_0)$. The
above calculation shows that the limit $u_1$ captures all the energy
in the oscillations of the homogenization corrector $v_{1\eps}$ in the
sense that the limit of the $L^2$ norm $\E\{v_{1\eps}^2(x,v)\}$ is
equal to the $L^2$ norm $\E\{u_1^2(x,v)\}$.

In higher dimensions $d\geq4$, as in the case of homogenization in
periodic media, some energy is lost while passing to the (weak) limit.
The corrector $u_{1\eps}=\eps^{-\frac d2}v_{1\eps}$ converges weakly and in
distribution to the limit $u_1$. However, while the energy of the
limiting corrector is $\eps^{\frac d2}
(\E\{\|u_1\|^2_{L^2(D}\}^{\frac12}$, the energy of the true corrector
$(\E\{\|v_{1\eps}\|^2_{L^2(D}\}^{\frac12}$ is of order $O(\eps^2)$ for
$d\geq5$ and of order $O(\eps^2|\ln\eps|^{\frac12})$ for $d=4$.
Most of the energy of the correctors is lost in passing from
$u_{1\eps}$ to its weak limit $u_1$.

%
\section{Correctors for one-dimensional elliptic problems}
\label{sec:1dbrdypbs}
%
In this section, we consider the homogenization of the following
one-dimensional elliptic problems:
\begin{equation}
  \label{eq:1dellipt}
  \begin{array}{l}
  -\dr{}x a_\eps(x,\omega) \dr{}x u_\eps + (q_0+q_\eps(x,\omega)) u_\eps =
   \rho_\eps(x,\omega) f(x) ,\qquad x\in D=(0,1),\\
  u_\eps(0)=u_\eps(1)=0.
  \end{array}
\end{equation}
We consider homogeneous Dirichlet conditions to simplify the
presentation. The coefficients $a_\eps(x,\omega)$ and
$\rho_\eps(x,\omega)$ are uniformly bounded from above and below:
$0<a_0\leq a_\eps(x,\omega), \rho_\eps(x,\omega)\leq a_0^{-1}$.  The
(deterministic) absorption term $q_0$ is assumed to be a non-negative
constant. The generalization to a non-negative smooth function
$q_0(x)$ can be done.

We assume that $a_\eps(x,\omega)=a(\frac x\eps,\omega)$,
$q_\eps(x,\omega)=q(\frac x\eps,\omega)$, and
$\rho_\eps(x,\omega)=\rho(\frac x\eps,\omega)$, where $a(x,\omega)$,
$q(x,\omega)$, and $\rho(x,\omega)$ are strictly stationary processes
on an abstract probability space $(\Omega,{\cal F},\P)$. We will modify
the mean-zero process $q_\eps(x,\omega)$ as in the preceding section
and assume here to simplify that $q(x,\omega)$ is bounded $\P-$a.s.
We also assume that the cross-correlations $R_{fg}(\bx)$
are integrable for $\{f,g\}\in\{a,q,\rho\}$, where
\begin{equation}
  \label{eq:crosscorr}
  R_{fg}(\bx)=\E\{f(\by,\omega)g(\by+\bx,\omega)\}.
\end{equation}
We also assume that the coefficients are jointly strongly mixing in
the sense of \eqref{eq:strongmixing}, where for two Borel sets $A$ and
$B$ in $\Rm^d$, we denote by ${\cal F}_A$ and ${\cal F}_B$ the
$\sigma$-algebras generated by the random fields $a(\bx,\omega)$,
$q(\bx,\omega)$, and $\rho(\bx,\omega)$. We still assume that the
$\rho$-mixing coefficient $\varphi(r)$ is integrable and such that
$\varphi^{\frac12}$ is also integrable.

In the case where $q_\eps=0$ and $\rho_\eps=0$, the corrector to the
homogenization limit $u_0$ has been considered in \cite{BP-AA-99}. For
general sufficiently mixing coefficients $a_\eps$ with positive
variance $\sigma^2=2\int_0^\infty \E\{a(0)a(t)\}dt>0$, we obtain that
$u_\eps-u_0$ is of order $\sqrt\eps$ and converges in distribution to
a Gaussian process. This section aims at generalizing the result to
\eqref{eq:1dellipt} using the results of the preceding section and a
change of variables based on harmonic coordinates \cite{Koz-MSb-79}.

Let us introduce the change of variables
\begin{equation}
  \label{eq:zeps}
  z_\eps(x) = a^*\dint_0^x \dfrac{1}{a_\eps(t)} dt,\qquad
  \dr{z_\eps}x = \dfrac{a^*}{a_\eps(x)},\qquad
   a^*=\dfrac{1}{\E\{a^{-1}\}},
\end{equation}
and $\tilde u_\eps(z)=u_\eps(x)$. Then we find, with $x=x(z_\eps)$
that
\begin{equation}
  \label{eq:tildueps}
  \begin{array}{l}
  -(a^*)^2 \drr{}z \tilde u_\eps + a^*q_0 \tilde u_\eps
   + a_\eps [(1-a_\eps^{-1}a^*)q_0+q_\eps] \tilde u_\eps
   = a_\eps  \rho_\eps f,\qquad 0<z<z_\eps(1)\\
  \tilde u_\eps(0)=\tilde u_\eps(z_\eps(1))=0.
  \end{array}
\end{equation}
Let us introduce the following Green's function
\begin{equation}
  \label{eq:GF}
  \begin{array}{l}
  -a^* \drr{}x G(x,y;L) +q_0 G(x,y;L)= \delta(x-y) \\
   G(0,y;L)=G(L,y;L)=0.
  \end{array}
\end{equation}
Then, defining 
\begin{equation}
  \label{eq:tildeqeps}
  \tilde q_\eps(x,\omega)=(1-a_\eps^{-1}(x,\omega)a^*)q_0+q_\eps(x,\omega),
\end{equation}
we find that 
\begin{displaymath}
  \begin{array}{rcl}
 \tilde u_\eps(z) &=& \dint_0^{z_\eps(1)} G(z,y;z_\eps(1))
     (\rho_\eps f-\tilde  q_\eps \tilde u_\eps)(x(y)) 
   \dfrac{a_\eps}{a^*}(x(y)) dy,\\
  u_\eps(x) &=& \dint_0^1 G(z_\eps(x),z_\eps(y);z_\eps(1))
          (\rho_\eps f-\tilde q_\eps u_\eps)(y) dy.
  \end{array}
\end{displaymath}
We recast the above equation as 
\begin{equation}
  \label{eq:uepsint0}
  u_\eps(x,\omega) = {\cal G}_\eps (\rho_\eps f-\tilde q_\eps u_\eps),\qquad
   {\cal G}_\eps u(x) = \dint_0^1 G(z_\eps(x),z_\eps(y);z_\eps(1)) u(y)dy.
\end{equation}
After one more iteration, we obtain the following integral equation:
\begin{equation}
  \label{eq:intueps}
  u_\eps = {\cal G}_\eps \rho_\eps f - 
   {\cal G}_\eps \tilde q_\eps {\cal G}_\eps \rho_\eps f
   + {\cal G}_\eps \tilde q_\eps {\cal G}_\eps \tilde q_\eps u_\eps.
\end{equation}
Since $a_0a^*x\leq z_\eps(x,\omega)\leq a^*a_0^{-1}x$ $\P-$a.s., the
Green's operator ${\cal G}_\eps$ is bounded $\P-$a.s. and the results
of Lemma \ref{lem:boundgqgq} generalize to the case where the operator
${\cal G}_\eps$ replaces ${\cal G}$. As in \eqref{eq:pbigradius}, we
thus modify $\tilde q_\eps$ (i.e. we modify $a_\eps$ and $q_\eps$) on a
set of measure $\eps$ so that $\|{\cal G}_\eps \tilde q_\eps {\cal
  G}_\eps \tilde q_\eps\|\leq r<1$ and assume that [H3] holds.

Let us introduce the notation
\begin{equation}
  \label{eq:notation1}
  \rho_\eps = \bar\rho + \delta \rho_\eps,\,\bar\rho=\E\{\rho\},\qquad
  {\cal G}_\eps = {\cal G} + \delta {\cal G}_\eps, \quad 
   {\cal G} =\E\{{\cal G}_\eps\},\qquad u_0 = {\cal G} \bar\rho f.
\end{equation}
We also define 
\begin{equation}
  \label{eq:deltazeps}
  \delta z_\eps(x)=z_\eps(x)-x = \dint_0^x b\Big(\dfrac{t}\eps\Big) dt,
  \qquad b(t,\omega) = \dfrac{a^*}{a(t,\omega)}-1.
\end{equation}
We first obtain the
\begin{lemma}
  \label{lem:bdgmgeps}
  We have that
  \begin{equation}
    \label{eq:bddeltazeps}
    \E \{|\delta z_\eps(x) \delta z_\eps(y)|\} \lesssim \eps,\qquad
     0\leq x,y \leq 1.
  \end{equation}
  The operator ${\cal G}_\eps$ may be decomposed as 
  \begin{equation}
    \label{eq:decGeps}
    {\cal G}_\eps = {\cal G} + {\cal G}_{1\eps} + {\cal R}_\eps, 
  \end{equation}
  where 
  \begin{equation}
    \label{eq:Gaeps}
    {\cal G}_{1\eps}f(x) = \dint_0^1  \Big(\delta z_\eps(x)\pdr{}x
   +  \delta z_\eps(y)\pdr{}y+  \delta z_\eps(1)\pdr{}L\Big)G(x,y;1)
         f(y) dy.
  \end{equation}
  We also have the following estimates
  \begin{equation}
    \label{eq:estimG1epsReps}
    \E\{\|{\cal G}_{1\eps}\|^2\}\lesssim \eps,\qquad
    \E\{\|{\cal R}_\eps\|\} \lesssim \eps.
  \end{equation}
\end{lemma}
\begin{proof}
  We first use the fact that
  \begin{displaymath}
    \E \{|\delta z_\eps(x) \delta z_\eps(y)|\}
    \leq \Big(\E \{(\delta z_\eps(x) \delta z_\eps(y))^2\}\Big)^{\frac12}.
  \end{displaymath}
  Denoting by $b_\eps(x,\omega)=b(\frac t\eps,\omega)$, we have to
  show that
  \begin{displaymath}
      \E \Big\{ \dint_0^x\dint_0^x\dint_0^y\dint_0^y 
     b_\eps(z_1)b_\eps(z_2) b_\eps(z_3)b_\eps(z_4) d[z_1z_2z_3z_4]\Big\}
    \lesssim \eps^2.
  \end{displaymath}
  Now using the mixing property of the mean-zero field $b_\eps$ and
  the integrability of $\varphi^{\frac12}(r)$, we obtain the result
  using \eqref{eq:bdq1234} as in the proof of Lemma \ref{lem:gqgqgf}.
  
  The integral defining ${\cal G}_\eps$ is split into two
  contributions, according as $y<x$ or $y>x$. On these two intervals,
  $G(x,y;L)$ is twice differentiable, and we thus have the expansion
  \begin{displaymath}
    G(z_\eps(x),z_\eps(y);z_\eps(1))=G(x,y;1)
   + \Big(\delta z_\eps(x)\pdr{}x
   +  \delta z_\eps(y)\pdr{}y+  \delta z_\eps(1)\pdr{}L\Big)G(x,y;1)
   + r_\eps,
  \end{displaymath}
  where the Lagrange remainder
  $r_\eps=r_\eps(x,z_\eps(x),y,z_\eps(y),z_\eps(1))$ is quadratic in
  the variables $(\delta z_\eps(x),\delta z_\eps(y),\delta z_\eps(1))$
  and involves second-order derivatives of $G(x,y;1)$ at points
  $(\xi,\zeta,L)$ between $(x,y;1)$ and
  $(z_\eps(x),z_\eps(y);z_\eps(1))$. 
  
  From \eqref{eq:bddeltazeps} and the fact that second-order
  derivatives of $G$ are $\P-$a.s. uniformly bounded on each interval
  $y<x$ and $y>x$ (we use here again the fact that $a_0a^*x\leq
  z_\eps(x,\omega)\leq a^*a_0^{-1}x$ $\P-$a.s.), we thus obtain that
  $\E\{|r_\eps(.)|\}\lesssim \eps$.  This also shows the bound for
  $\E\{\|{\cal R}_\eps\|\}$ in \eqref{eq:estimG1epsReps}. The bound
  for $\E\{\|{\cal G}_{1\eps}\|^2\}$ is obtained similarly.
\end{proof}

Because we have assumed that $\tilde q_\eps$ and $\rho_\eps$ were
bounded $\P-$a.s., we can replace ${\cal G}_\eps$ by ${\cal G}+{\cal
  G}_{1\eps}$ in \eqref{eq:intueps} up to an error of order $\eps$ in
$L^1(\Omega;L^2(D))$. The case of $q_\eps$ and $\rho_\eps$ bounded on
average would require to address their correlation with $r_\eps$
defined in the proof of the preceding lemma. This is not considered
here.

We recast \eqref{eq:intueps} as
\begin{equation}
  \label{eq:uepsmu0}
  u_\eps-u_0 = ({\cal G}_\eps \rho_\eps-{\cal G}\bar\rho)f
  -{\cal G}_\eps \tilde q_\eps{\cal G}_\eps \rho_\eps f
  +  {\cal G}_\eps \tilde q_\eps{\cal G}_\eps \tilde q_\eps(u_\eps-u_0)
  +{\cal G}_\eps \tilde q_\eps{\cal G}_\eps \tilde q_\eps{\cal G}f.
\end{equation}
Because $G(z_\eps(x),z_\eps(y);z_\eps(1))$ and $\rho_\eps$ are
uniformly bounded $\P-$a.s., the proof of Lemma \ref{lem:boundgqgq}
generalizes to give us that 
\begin{equation}
  \label{eq:bdopsoeps}
  \E\{ \|{\cal G}_\eps\tilde q_\eps{\cal
  G}_\eps\tilde q_\eps \|^2\} + 
   \E\{\|  {\cal G}_\eps \tilde q_\eps{\cal G}_\eps \rho_\eps f \|^2\}
   +\E\{\| ({\cal G}_\eps \rho_\eps-{\cal G}\bar\rho)f \|^2\}
   \lesssim\eps.
\end{equation}
So far, since moreover $\|{\cal G}_\eps \tilde q_\eps {\cal G}_\eps
\tilde q_\eps\|\leq r<1$, we have thus obtained the following result:
\begin{lemma}
  \label{lem:estcorr1d}
  Let $u_\eps$ be the solution to the heterogeneous problem
  \eqref{eq:1dellipt} and $u_0$ the solution to the corresponding
  homogenized problem. Then we have that 
  \begin{equation}
    \label{eq:esterror}
    \big(\E\{\|u_\eps-u_0\|^2\}\big)^{\frac12} \lesssim \sqrt\eps \|f\|.
  \end{equation}
\end{lemma}
The estimate \eqref{eq:bdaepshelm} with $d=1$ is thus verified in the
context of the elliptic equation \eqref{eq:1dellipt}.  As a
consequence, we find that $\E\{\|u_\eps-u_0\|^2\}\lesssim\eps$ so that
by Cauchy Schwarz and \eqref{eq:bdopsoeps},
\begin{displaymath}
  \E\{\|{\cal G}_\eps \tilde q_\eps{\cal G}_\eps \tilde q_\eps(u_\eps-u_0)\|
   \} \lesssim \eps.
\end{displaymath}
It remains to exhibit the term of order $\sqrt\eps$ in $u_\eps-u_0$.
Let us introduce the decomposition
\begin{eqnarray}
  \label{eq:decueupsmu0}
   &&u_\eps-u_0 =  \Big[{\cal G}_{1\eps}\bar\rho + {\cal G}\delta \rho_\eps
   - {\cal G}\tilde q_\eps {\cal G}\bar\rho\Big] f + s_\eps,\\
   \nonumber
  &&s_\eps =  (\delta {\cal G}_\eps\delta \rho_\eps+{\cal R}_\eps \bar\rho) f
  - ({\cal G}_\eps \tilde q_\eps{\cal G}_\eps \rho_\eps-
    {\cal G}\tilde q_\eps {\cal G}\bar\rho)f
  +{\cal G}_\eps \tilde q_\eps{\cal G}_\eps \tilde q_\eps(u_\eps-u_0)
  +{\cal G}_\eps \tilde q_\eps{\cal G}_\eps \tilde q_\eps{\cal G}f.
\end{eqnarray}
\begin{lemma}
  \label{lem:seps}
  Let $f\in L^2(D)$.  We have
  \begin{equation} \label{eq:bdseps}
    \E\{\|s_\eps\|\} \lesssim \eps \|f\|.
  \end{equation}
\end{lemma}
\begin{proof}
  Because $G(z_\eps(x),z_\eps(y);z_\eps(1))$ is uniformly bounded
  $\P-$a.s., the proof of Lemma \ref{lem:gqgqgf} generalizes to show
  that $\E\{\|{\cal G}_\eps \tilde q_\eps{\cal G}_\eps \tilde
  q_\eps{\cal G}f\|^2\}\lesssim \eps^2\|f\|^2$. We already know that
  $\E\{\|{\cal R}_\eps\|\}\lesssim \eps$. It remains to address the
  terms $I_1={\cal G}_{1\eps} \delta\rho_\eps f$, $I_2={\cal G}\tilde
  q_\eps{\cal G}_{1\eps}\rho_\eps f$, $I_3={\cal G}_{1\eps}\tilde
  q_\eps{\cal G}\rho_\eps f$, and $I_4={\cal G}\tilde q_\eps{\cal
    G}\delta \rho_\eps f$.
  
  Because $\rho_\eps$ is uniformly bounded $\P-$a.s., the first three
  terms are handled in a similar way. Let us consider
  $\E \{I_1^2\}$, which is bounded by a finite number (three here) of 
  operators of the form
  \begin{displaymath}
    \E\{\dint_{D^3} \delta z_\eps(v_1(x,y)) H(x,y)
      \delta z_\eps(v_2(x,\zeta)) H(x,\zeta) 
     \delta\rho_\eps(y) \delta\rho_\eps(\zeta) f(y)f(\zeta)dxdyd\zeta \},
  \end{displaymath}
  where $v_k(x,y)$ is either $x$, $y$, or $1$ for $k=1,2$, and
  $H(x,y)$ is a uniformly bounded function. Using the definition of
  $\delta z_\eps$, we recast the above integral as
  \begin{displaymath}
    \dint_{D^3} \dint_0^{v_1}\dint_0^{v_2}
    \E\{b_\eps(t_1)b_\eps(t_2)\delta \rho_\eps(y)\delta\rho_\eps(\zeta)\}
     dt_1 dt_2 H(x,y) H(x,\zeta) f(y)f(\zeta) dxdyd\zeta.
  \end{displaymath}
  Using \eqref{eq:bdq1234}, we see that the above integral is bounded 
  by terms of the form
  \begin{displaymath}
    \dint_{D^3} \dint_0^{v_1}\dint_0^{v_2}
     \varphi^{\frac12}\Big(\dfrac{u_1-u_2}\eps\Big)
     \varphi^{\frac12}\Big(\dfrac{u_3-u_4}\eps\Big)
     dt_1 dt_2 | H(x,y) H(x,\zeta) | |f(y)||f(\zeta)| dxdyd\zeta,
  \end{displaymath}
  where $(u_1,u_2,u_3,u_4)=(u_1,u_2,u_3,u_4)(t_1,t_2,y,\zeta)$ is an
  arbitrary (fixed) permutation of $(t_1,t_2,y,\zeta)$. Because
  $\varphi(r)$ is integrable, the Cauchy Schwarz inequality shows that
  the above term is $\lesssim\eps^2\|f\|^2$.  The term $\E\{I_4^2\}$
  is given by
  \begin{displaymath}
    \E\Big\{\dint_{D^4} G(x,y)G(x,\zeta) \tilde q_\eps(y)\tilde q_\eps(\zeta)
       G(y,z)  G(\zeta,\xi) \delta\rho_\eps(z) \delta\rho_\eps(\xi)
     f(z)f(\xi) d[xyz\zeta\xi]\Big\}.
  \end{displaymath}
  Since $G(x,y)$ is uniformly bounded on $D$, we again use
  \eqref{eq:bdq1234} as above to obtain a bound of the form
  $\eps^2\|f\|^2$.
\end{proof}
It remains to
analyze the convergence of the contribution $[{\cal G}_{1\eps}\bar\rho +
{\cal G}\delta \rho_\eps - {\cal G}\tilde q_\eps {\cal G}\bar\rho]f$.

As in \eqref{eq:u1eps1dhelm}, we define
\begin{equation}
  \label{eq:u1epsellip}
  u_{1\eps}(x,\omega)= \dfrac 1{\sqrt\eps}\Big[{\cal G}_{1\eps}\bar\rho +
{\cal G}\delta \rho_\eps - {\cal G}\tilde q_\eps {\cal G}\bar\rho\Big]f (x).
\end{equation}
We recast the above term as 
\begin{equation}
  \label{eq:kernelsu1eps}
  u_{1\eps}(x,\omega)=\dfrac1{\sqrt\eps}\dint_0^1 \Big[
    b\Big(\dfrac t\eps\Big) H_b(x,t) + 
    \delta\rho\Big(\dfrac t\eps\Big) H_\rho(x,t) -
    \tilde q\Big(\dfrac t\eps\Big) H_q(x,t)\Big]dt,
\end{equation}
with 
\begin{equation}
  \label{eq:kernels1d}
  \begin{array}{rcl}
    H_b(x,t)&=&\dint_0^1 \Big[\chi_x(t) \pdr{}x G(x,y;1)
     + \chi_y(t)  \pdr{}y G(x,y;1) + \pdr{}L G(x,y;1)\Big]
    \bar\rho f(y) dy  \!\!\!\!\!\!\!\!\!\!\\
   H_\rho(x,t) &=& G(x,t) f(t) \\
   H_q(x,t) &=& G(x,t) \dint_0^1 G(t,z) f(z) dz.
  \end{array}
\end{equation}
where $\chi_x(t)=1$ if $0<t<x$ and vanishes otherwise. We have the 
following result.
\begin{theorem}
  \label{thm:cvellip1d}
  Let $f\in L^\infty(0,1)$.  The process $u_{1\eps}(x,\omega)$
  converges weakly and in distribution in the space of continuous
  paths ${\cal C}(D)$ to the limit $u_1(x,\omega)$ given by
  \begin{equation}
    \label{eq:u1ellip1d}
    u_1(x,\omega)= \dint_0^1 \sigma(x,t) dW_t,
  \end{equation} 
  where $W_t$ is standard Brownian motion and 
  \begin{equation}
    \label{eq:sigmaellip1d}
    \begin{array}{rcl}
    \sigma^2(x,t) &=& 2\dint_0^\infty 
        \E\{F(x,t,0)F(x,t,\tau)\} d\tau,\\
     F(x,t,\tau)&=&H_b(x,t) b(\tau)+ H_{\rho}(x,t) \delta \rho(\tau)
      - H_q(x,t)\tilde q(\tau).
   \end{array}
  \end{equation}
  As a consequence, the corrector to homogenization thus satisfies
  that:
  \begin{equation}
    \label{eq:cvuepsellip1d}
    \dfrac{u_\eps-u_0}{\sqrt\eps}(x) \xrightarrow{\,\rm dist. \,}
     u_1(x,\omega),\quad \mbox{ as }\eps\to0,
  \end{equation}
   in the space of integrable paths $L^1(D)$.
\end{theorem}
We may recast $u_{1\eps}(x,\omega)$ as
\begin{equation}
    \label{eq:decu1eps}
    u_{1\eps}(x,\omega)=\dsum_{k=1}^3 \dfrac{1}{\sqrt\eps}
     \dint_D p_k(\dfrac t\eps) H_k(x,t)dt,
\end{equation}
where the $p_k$ are mean-zero processes and the kernels $H_k(x,t)$ are
given in \eqref{eq:kernels1d}. The corrector in \eqref{eq:u1ellip1d}
may then be rewritten as
\begin{equation}
  \label{eq:u1ellip1d-2}
  u_1(x) = \dsum_{k=1}^3 \dint_D \sigma_k(x,t) dW_t^j,
\end{equation}
with three correlated standard Brownian motions such that
\begin{equation}
  \label{eq:corrBM}
  dW^j_t dW^k_t = \rho_{jk} dt, 
\end{equation}
where we have defined
\begin{equation}
  \label{eq:sigmarho}
  \begin{array}{rcl}
 \sigma_k(x,t) &=& H_k(x,t) \sqrt2\Big(\dint_0^\infty \E\{p_k(0)p_k(\tau)\}
    d\tau\Big)^{\frac12} \\[2mm]
  \rho_{jk} &=& \dfrac{\dint_0^\infty \E\{p_j(0)p_k(\tau)+p_k(0)p_j(\tau)\}
   d\tau}{2\Big(\dint_0^\infty \E\{p_j(0)p_j(\tau)\}d\tau
        \dint_0^\infty \E\{p_k(0)p_k(\tau)\}d\tau\} \Big)^{\frac12}}.
  \end{array}
\end{equation}
That \eqref{eq:u1ellip1d} and \eqref{eq:u1ellip1d-2} are equivalent
comes from the straightforward calculation that both processes are
mean zero Gaussian processes that have the same correlation function.
The new equation \eqref{eq:u1ellip1d-2} shows more clearly the
linearity of the $\sigma_k$, whence $u_1(x)$, with respect to the
source term $f(x)$.

\begin{proof}
  We recast $u_{1\eps}(x,\omega)$ as
  \begin{displaymath}
    u_{1\eps}(x,\omega)=\dsum_{k} \dfrac{1}{\sqrt\eps}
     \dint_D q_k(\dfrac t\eps) H_k(x,t)dt,
  \end{displaymath}
  with a different decomposition as in \eqref{eq:decu1eps}, where the
  $q_k$ are mean-zero processes and the kernels $H_k(x,t)$ are given
  implicitly in \eqref{eq:kernels1d}.  We verify that we can choose
  the terms $H_k(x,t)$ in the above decomposition so that all of them
  are uniformly (in $t$) Lipschitz in $x$, except for one term, say
  $H_1(x,t)$, which is of the form
  \begin{displaymath}
    H_1(x,t) = \chi_x(t) L_1(x,t),\qquad
    L_1(x,t) = \dint_0^1 \pdr{}x G(x,y;1) \bar\rho f(y) dy,
  \end{displaymath}
  where $L_1(x,t)$ is uniformly (in $t$) Lipschitz in $x$.  This
  results from the fact that $G(x,y;1)$ is Lipschitz continuous and
  that its partial derivatives are bounded and piecewise Lipschitz
  continuous; we leave the tedious details to the reader.
  
  Because of the presence of the term $H_1(x,t)$ in the above
  expression, it is not sufficient to consider second-order moments of
  $u_{1\eps}$ as in the proof of Thm. \ref{thm:cv1dhelm}. Rather, we
  consider fourth-order moments as follows:
  \begin{displaymath}
    \begin{array}{ll}
        \E\{|u_{1\eps}(x,\omega)-u_{1\eps}(\xi,\omega)|^4\}
   = &\dfrac{1}{\eps^2} \dsum_{k_1,k_2,k_3,k_4} \dint_{D^4}
    \E\{q_{k_1}(\dfrac {t_1}\eps)q_{k_2}(\dfrac {t_2}\eps)
    q_{k_3}(\dfrac {t_3}\eps)q_{k_4}(\dfrac {t_4}\eps)\}\times\\ \qquad&
   \prod\limits_{m=1}^4 
   (H_{k_m}(x,t_m)-H_{k_m}(\xi,t_m)) dt_1dt_2dt_3dt_4.
    \end{array}
  \end{displaymath}
  Using the mixing condition of the processes $q_k$ and Lemma
  \ref{lem:mixingfourpoints} (where each $q$ in \eqref{eq:bdq1234} may
  be replaced by $q_k$ without any change in the result), we obtain
  that $\E\{|u_{1\eps}(x,\omega)-u_{1\eps}(\xi,\omega)|^4\}$ is
  bounded by a sum of terms of the form
  \begin{displaymath}
    \dfrac{1}{\eps^2}\dint_{D^4} \varphi^{\frac12}(\frac{t_2-t_1}\eps)
       \varphi^{\frac12}(\frac{t_4-t_3}\eps)\prod\limits_{m=1}^4 
   (H_{k_m}(x,t_m)-H_{k_m}(\xi,t_m)) dt_1dt_2dt_3dt_4,
  \end{displaymath}
  whence is bounded by terms of the form
  \begin{displaymath}
    \Big( \dfrac{1}{\eps} \dint_{D^2} 
    \varphi^{\frac12}(\frac{t_2-t_1}\eps)\prod\limits_{m=1}^2
   (H_{k_m}(x,t_m)-H_{k_m}(\xi,t_m)) dt_1dt_2 \Big)^2.
  \end{displaymath}
  When all the kernels $H_{k_m}$ are Lipschitz continuous, then the
  above term is of order $|x-\xi|^4$.  The largest contribution is
  obtained when $k_1=k_2=1$ because $H_1(x,t)$ is not uniformly
  Lipschitz continuous. We now concentrate on that contribution. We
  recast
  \begin{displaymath}
    H_1(x,t)-H_1(\xi,t) = (\chi_x(t)-\chi_\xi(t))L_1(x,t) +
    \chi_\xi(t)(L_1(x,t)-L_1(\xi,t)).
  \end{displaymath}
  Again, the largest contribution to the fourth moment of $u_{1\eps}$
  comes from the term $(\chi_x(t)-\chi_\xi(t))L_1(x,t)$ since
  $L_1(x,t)$ is Lipschitz continuous. Assuming that $x\geq\xi$ without
  loss of generality, we calculate that 
  \begin{displaymath}
    \begin{array}{ll}
    \dint_{D^2} (\chi_x(t)-\chi_\xi(t))L_1(x,t)
     (\chi_{\xi}(s)-\chi_\xi(s))L_1(\xi,s) \dfrac{1}{\eps}
    \varphi^{\frac12}(\frac{t-s}\eps) dtds \\
   = \dint_{\xi}^x\dint_{\xi}^x L_1(x,t)L_1(\xi,s) \dfrac{1}{\eps}
    \varphi^{\frac12}(\frac{t-s}\eps) dtds \lesssim (x-\xi),
    \end{array}
  \end{displaymath}
  since $\varphi^{\frac12}$ is integrable. Note that this term
  is not of order $|\xi-x|^2$. Nonetheless, we have shown that 
  \begin{displaymath}
     \E\{|u_{1\eps}(x,\omega)-u_{1\eps}(\xi,\omega)|^4\} \lesssim
    |\xi-x|^2,
  \end{displaymath}
  so that we can apply the Kolmogorov criterion in Prop. \ref{prop:2}
  with $\nu=2$, $\beta=4$, and $\delta=1$. This concludes the proof of
  tightness of $u_{1\eps}(x,\omega)$ as a process with values in the
  space of continuous functions ${\cal C}(D)$.

  \medskip
  
  It remains to verify step (a) of Prop. \ref{prop:2}.  The
  finite-dimensional distributions are treated as in the proof of Thm.
  \ref{thm:cv1dhelm} and are replaced by the analysis of random
  integrals of the form:
  \begin{displaymath}
  \dfrac1{\sqrt\eps}\dint_0^1 \Big[
    b\Big(\dfrac t\eps\Big) m_b(t) + 
    \delta\rho\Big(\dfrac t\eps\Big) m_\rho(t) + 
    \tilde q\Big(\dfrac t\eps\Big) m_q(t)\Big]dt.
  \end{displaymath}
  The functions $m$ are continuous and can be approximated by $m_h$
  constant on intervals of size $h$ so that we end up with $M$
  independent (in the limit $\eps\to0$) variables of the form:
  \begin{displaymath}
  \dfrac{\sqrt h}{\sqrt N} \dsum_{j=1}^N
  m_{bh} b_j + m_{\rho h} \delta \rho_j + m_{qh} \tilde q_j.
  \end{displaymath}
  It remains to apply the central limit theorem as in the proof of
  Thm.  \ref{thm:cv1dhelm}. The above random variable converges in
  distribution to
  \begin{displaymath}
  {\cal N}(0,h\sigma^2),\quad
  \sigma^2 = 2\dint_0^\infty \E\{(m_{bh} b + m_{\rho h} \delta \rho
         + m_{qh} \tilde q)(0)(m_{bh} b + m_{\rho h} \delta \rho
         + m_{qh} \tilde q)(t)\} dt.
  \end{displaymath}
  This concludes our analysis of the convergence in distribution of
  $u_{1\eps}$ to its limit in the space of continuous paths ${\cal
    C}(D)$. The convergence of $u_\eps-u_0$ follows from the bound
  \eqref{eq:bdseps}.
\end{proof}

\section{Correctors for spectral problems}
\label{sec:spectral}

\subsection{Abstract convergence result}
\label{sec:abstractspectral}

For $\omega\in\Omega$, let $A_\eta(\omega)$ be a sequence of bounded
(uniformly in $\omega$ $\P-$a.s. and in $\eta>0$), compact,
self-adjoint operators, converging to a deterministic, compact,
self-adjoint, operator $A$ as $\eta\to0$ in the sense that the
following error estimate holds:
\begin{equation}
  \label{eq:errorAeps}
  \E\|A_\eta(\omega) - A\|^p\lesssim \eta^p,\qquad 
   \mbox{ for some } 1\leq p<\infty,
\end{equation}
where $\|A_\eta(\omega) - A\|$ is the $L^2(D)$ norm and $D$ is an open
subset of $\Rm^d$.

The operators $A$ and $\P-$a.s. $A_\eta(\omega)$ admit the spectral
decompositions $(\lambda_n,u_n)$ and $(\lambda_n^\eta,u_n^\eta)$,
where the real-valued eigenvalues are ordered in decreasing values of
their absolute values and counted $m_n$ times, where $m_n$ is their
multiplicity.

For $\lambda_n$, let $\mu_n$ be (one of) the closest eigenvalue of $A$
that is different from $\lambda_n$. Let us then define the distance:
\begin{equation}
  \label{eq:distdn}
  d_n = \dfrac{|\lambda_n-\mu_n|}2.
\end{equation}
Following \cite{Kato66}, we analyze the spectrum of $A_\eta$ in the
vicinity of $\lambda_n$. Let $\Gamma$ be the circle of center
$\lambda_n$ and radius $d_n$ in the complex plane and let
$R(\zeta,A)=(A-\zeta)^{-1}$ be the resolvent of $A$ defined for all
complex numbers $\zeta\not\in \sigma(A)$, the spectrum of $A$. The
projection operator onto the spectral components of $B$ inside the
curve $\Gamma$ is defined by
\begin{equation}
  \label{eq:Pnspectral}
  P_n[B] = -\dfrac{1}{2\pi i}\dint_\Gamma R(\zeta,B)d\zeta.
\end{equation}
Note that for all $\zeta\in\Gamma$, we have that
$R(\zeta,A)P_n[A]=(\lambda_n-\zeta)^{-1}$. We can then prove the
following result:
\begin{proposition}
  \label{prop:cvspect}
  Let $A_\eta$ and $A$ be the operators described above and let
  $\lambda_n$ be fixed.  Then, for $\eta$ sufficiently small, there
  are exactly $m_n$ eigenvalues $\lambda_n^\eta$ of $A_\eta$ inside
  the circle $\Gamma$.  Moreover, we have the following estimates:
  \begin{equation}
    \label{eq:evspetral}
    \E \{ |\lambda_n - \lambda_n^\eta|^p \} + 
     \E \{ |\|u_n^\eta-u_n\|^p \} \lesssim 
      \dfrac{\eta^p}{d_n^p}\wedge 1,
  \end{equation}
  for a suitable labeling of the eigenvectors $u_n^\eta$ of $A^\eta$
  associated to the eigenvalues $\lambda_n^\eta$.
\end{proposition}
\begin{proof}
  It follows from \cite[Theorem IV.3.18]{Kato66} that for those
  realizations $\omega$ such that $\|A_\eta(\omega)-A\|<d_n$, then
  there are exactly $m_n$ eigenvalues of $A_\eta$ in the
  $d_n-$vicinity of $\Gamma$. Since this also holds for every
  $\lambda_m$ such that $d_m>d_n$, we can index the eigenvalues of
  $A_\eta$ as the eigenvalues of $A$. Moreover, 
  \begin{displaymath}
    |\lambda_n^\eta(\omega)-\lambda_n| \leq \|A_\eta(\omega)-A\|.
  \end{displaymath}
  For those realizations $\omega$ such that $\|A_\eta(\omega)-A\|\geq
  d_n$, we choose $m_n$ eigenvalues of $A_\eta(\omega)$ arbitrarily
  among the eigenvalues that have not been chosen in the
  $d_m-$vicinity of $\lambda_m$ for $|\lambda_m|>|\lambda_n|$.

  For all realizations, we thus obtain that 
  \begin{displaymath}
     |\lambda_n^\eta(\omega)-\lambda_n| \lesssim 
     \dfrac{\|A_\eta(\omega)-A\|}{d_n}.
  \end{displaymath}
  It remains to take the $p$th power and average the above expression
  to obtain the first inequality of the proposition.
  
  In order for the eigenvectors $u^\eta_n$ and $u_n$ to be close, we
  need to restrict the size of $\eta$ further. To make sure the
  eigenvectors are sufficiently close, we need to ensure that
  \begin{displaymath}
  P_n[A_\eta] - P_n[A] = \dfrac{-1}{2\pi i}\dint_\Gamma
   [R(\zeta,A_\eta)-R(\zeta,A)] d\zeta
   = \dfrac{1}{2\pi i} \dint_\Gamma R(\zeta,A_\eta)(A_\eta-A)
    R(\zeta,A) d\zeta,
  \end{displaymath}
  is sufficiently small. On the circle $\Gamma$ and for
  $\|A-A_\eta\|<d_n$, we verify that
  \begin{displaymath}
     \sup\limits_{\zeta\in\Gamma} \|R(\zeta,A)\| = \dfrac{1}{d_n},\qquad
    \sup\limits_{\zeta\in\Gamma} \|R(\zeta,A_\eta)\| \leq 
     \dfrac{1}{d_n-\|A-A_\eta\|},
  \end{displaymath}
  by construction of $d_n$ and by using
  $R^{-1}(\zeta,A_\eta)=R^{-1}(\zeta,A)+(A_\eta-A)$ and the triangle
  inequality
  \begin{displaymath}
    \|R^{-1}(\zeta,A_\eta)\| \geq \|R^{-1}(\zeta,A)\|- \|A_\eta-A\|
    \geq d_n - \|A_\eta-A\|.  
  \end{displaymath}
  Upon integrating the expression for
  $P_n[A_\eta]-P_n[A]$ on $\Gamma$, we find that 
  \begin{displaymath}
    \rho:=\|P_n[A_\eta]-P_n[A]\| \leq  \dfrac{\|A_\eta-A\|}{d_n-\|A_\eta-A\|} 
    \leq \dfrac{ 2 }{d_n}\|A_\eta-A\|<1 ,
  \end{displaymath}
  for $2\|A_\eta-A\|<d_n$.
  
  For self-adjoint operators $A$ and $A_\eta$, the above bound on the
  distance $\rho$ between the eigenspaces is sufficient to
  characterize the distance between the corresponding eigenvectors. We
  follow \cite[I.4.6 \& II.4.2]{Kato66} and construct the unitary
  operator
  \begin{equation}
    \label{eq:isomesp}
    U^\eta_n = \Big(I-(P_n[A_\eta]-P_n[A])^2\Big)^{-\frac12}
      \Big(P_n[A_\eta] P_n[A] + (I-P_n[A_\eta])(I-P_n[A])\Big).
  \end{equation}
  
  Let $u_{n,k}$, $1\leq k\leq m_n$ be all the eigenvectors associated
  to an eigenvalue $\lambda_n$, $n\geq1$. Then the eigenspace
  associated to $\lambda_n^\eta$ admits for an orthonormal basis the
  eigenvectors defined by \cite{Kato66}
  \begin{equation}
    \label{eq:uetan}
    u_{n,k}^\eta = U^\eta_n u_{n,k},\qquad 1\leq k\leq m_n.
  \end{equation}
  
  The relation \eqref{eq:isomesp} may be recast as
  \begin{displaymath}
    U^\eta_n = (I-R^\eta_n)\big(I+P_n[A_\eta](P_n[A_\eta]-P_n[A])
        + (P_n[A_\eta]-P_n[A])P_n[A_\eta]\big),
  \end{displaymath}
  where $\|R^\eta_n\|\lesssim \rho^2$. This shows that 
  \begin{displaymath}
    \|U^\eta_n-I\| \lesssim \rho\qquad \mbox{ and } \qquad
    \|u_{n,k}^\eta-u_{n,k} \| \lesssim \rho
    \lesssim \dfrac{1}{d_n} \|A_\eta-A\|,\quad 1\leq k\leq m_n,
  \end{displaymath}
  whenever $d_n^{-1}\|A_\eta-A\|<\mu$ for $\mu$ sufficiently small.
  When $d_n^{-1}\|A_\eta-A\|\geq\mu$, we find that
  $\|u_{n,k}-u_{n,k}^\eta\|\lesssim 2\mu\|A_\eta(\omega)-A\|/d_n$,
  where the vectors $u_{n,k}^\eta$ are constructed as an arbitrary
  orthonormal basis of the eigenspace associated to $\lambda_n^\eta$.
  Upon taking $p$th power and ensemble averaging, we obtain
  \eqref{eq:evspetral}.
\end{proof}

\subsection{Correctors for eigenvalues and eigenvectors}
\label{sec:correigv}

Let $(\lambda_n,u_n)$ be a solution of $Au_n=\lambda_n u_n$ and let
$\lambda^\eta_n$ and $u^\eta_n$ be the solution of $A_\eta u^\eta_n =
\lambda^\eta_n u^\eta_n$ defined in Proposition \ref{prop:cvspect}.
We assume that \eqref{eq:errorAeps} holds with $p=2$.

We calculate that
\begin{displaymath}
  \dfrac{\lambda^\eta_n-\lambda_n}{\eta} = \Big( u_n, \dfrac
   {A_\eta-A}{\eta} u_n \Big) + \dfrac 1\eta
   \Big( u_n^\eta-u_n , \big((A_\eta-\lambda^\eta_n)-(A-\lambda_n)\big)
    u_n\Big).
\end{displaymath}
The last term, which we denote by $r_n^\eta(\omega)$ is bounded by
$O(\eta)$ in $L^1(\Omega)$ using the results of Proposition
\ref{prop:cvspect} with $p=2$ and the Cauchy Schwarz inequality. Thus,
$r_n^\eta(\omega)$ converges to $0$ in probability.

Let us assume that the eigenvectors are defined on a domain
$D\subset\Rm^d$ and that for a smooth function $M(\bx)$, we have:
\begin{equation}
  \label{eq:abstcv}
  \Big( M(\bx), \dfrac{A_\eta-A}\eta u_n (\bx)\Big)  
    \xrightarrow{\,\rm dist. \,}
   \dint_{D^2} M(\bx) \sigma_n(\bx,\by) dW_{\by}d\bx 
   \qquad \mbox{ as } \eta\to0.
\end{equation}
Using this result, and provided that the eigenvectors $u_n(\bx)$ are
sufficiently smooth, we obtain that
\begin{equation}
  \label{eq:cvcorrspect}
   \dfrac{\lambda^\eta_n-\lambda_n}{\eta}  
    \xrightarrow{\,\rm dist. \,}
   \dint_{D^2} u_n(\bx) \sigma_n(\bx,\by) dW_{\by}d\bx
   := \dint_D \Lambda_n(\by)dW_{\by}\qquad 
   \mbox{ as } \eta\to0.
\end{equation}
The eigenvalue correctors are therefore Gaussian variables, which may
conveniently be written as a stochastic integral that is quadratic in
the eigenvectors since $\sigma_n(\bx,\by)$ is a linear functional of
$u_n$.  The correlations between different correctors may also
obviously be obtained as
\begin{equation}
  \label{eq:correigv}
  \E \Big\{ \dfrac{\lambda^\eta_n-\lambda_n}{\eta} 
        \dfrac{\lambda^\eta_m-\lambda_m}{\eta} \Big\}  
    \xrightarrow{\eta \to 0}
   \dint_D \Lambda_n(\bx) \Lambda_m(\bx) d\bx.
\end{equation}

Let us now turn to the corrector for the eigenvectors. Note that
\begin{displaymath}
  \|u_n-u_n^\eta\|^2 = 2(1-(u_n,u_n^\eta)),
\end{displaymath}
so that $(u_n,u_n^\eta)$ is equal to $1$ plus an error term of order
$O(\eta^2)$ on average. The construction of the eigenvectors in
\eqref{eq:uetan} show that $u_n-u_n^\eta$ is of order $O(\eta^2)$ in
the whole eigenspace associated to the eigenvalue $\lambda_n$.  It
thus remains to analyze the convergence properties of
$(u_n-u_n^\eta,u_m)$ for all $m\not=n$. A straightforward calculation
similar to the one obtained for the eigenvalue corrector shows that
\begin{displaymath}
  \Big(\dfrac{u_n^\eta-u_n}\eta, (A-\lambda_n)u_m \Big)
   = - \Big( \dfrac{(A_\eta-\lambda_n^\eta)-(A-\lambda_n)}\eta u_n,u_m\Big)
     - \dfrac{1}{\eta} ( (A_\eta-A)(u_n^\eta-u_n),u_m).
\end{displaymath}
The last term converges to $0$ in probability (and is in fact of order
$O(\eta)$ in $L^1(\Omega)$ as above). We thus find that 
\begin{equation}
  \label{eq:conveig}
  \Big(\dfrac{u_n^\eta-u_n}\eta, u_m\Big)  
    \xrightarrow{\,\rm dist. \,}
  \dfrac{1}{\lambda_n-\lambda_m}
   \dint_{D^2}u_m(\bx)\sigma_n(\bx,\by)dW_\by d\bx.
\end{equation}
The Fourier coefficients of the eigenvector correctors converge to
Gaussian random variables. As in the case of eigenvalues, it is
straightforward to estimate the cross-correlations of the Fourier
coefficients corresponding to (possibly) different eigenvectors.
\subsection{Applications to some specific problems}
\label{sec:applispectral}

The first application pertains to the following problem:
\begin{equation}
  \label{eq:Helmspect}
  A_\eps = (P(\bx,D)+q_\eps)^{-1},\qquad
  A = P(\bx,D)^{-1}.
\end{equation}
Lemma \ref{lem:estcorrHelm} and its corollary \eqref{eq:bdaepshelm}
show that \eqref{eq:errorAeps} holds with $p=2$ and $\eta=\eps^{\frac
  d2}$. The operators $A_\eps$ and $A$ are also compact and
self-adjoint for a large class of operators $P(\bx,D)$ which includes
the Helmholtz operator $P(\bx,D)=-\Delta + q_0(\bx)$.

Let $(\lambda_n^\eps,u_n^\eps)$ be the solutions of $\lambda^\eps
P_\eps u_\eps = u_\eps$ and $(\lambda_n,u_n)$ the solutions of
$\lambda Pu=u$. Then we find that 
\begin{equation}
  \label{eq:conveighelm}
  \dfrac{\lambda^\eps_n-\lambda_n}{\eps^{\frac d2}} 
    \xrightarrow{\,\rm dist. \,}
   -\lambda_n \sigma \dint_{D^2} u_n(\bx) G(\bx,\by) u_n(\by) dW_{\by} d\bx
  = -\lambda_n^2 \sigma \dint_D u_n^2(\by) dW_{\by},
\end{equation}
or equivalently, that for the eigenvalues of $P_\eps$ and $P$, we
have:
\begin{equation}
  \label{eq:cvinveighelm}
  \dfrac{(\lambda^\eps_n)^{-1}-\lambda_n^{-1}}{\eps^{\frac d2}} 
    \xrightarrow{\,\rm dist. \,}\sigma \dint_D u_n^2(\by) dW_{\by}.
\end{equation}
The Fourier coefficients of the eigenvectors satisfy similar expressions.

The second example is the one-dimensional elliptic equation
\eqref{eq:1dellipt}. Still setting $\eta=\eps^{\frac12}$,
we find that 
\begin{displaymath}
  \dfrac{A_\eps-A}{\sqrt\eps} u_n  \xrightarrow{\,\rm dist. \,}
  \dint_D \sigma_n(x,t) dW_t,
\end{displaymath}
where $\sigma_n(x,t)$ is defined in \eqref{eq:sigmaellip1d} with the
source term $f$ in \eqref{eq:kernels1d} being replaced by $u_n(x)$. The
operators $A_\eps$ and $A$ satisfy \eqref{eq:errorAeps} with $p=2$
thanks to Lemma \ref{lem:estcorr1d} and its corollary
\eqref{eq:bdaepshelm}. The expressions for the eigenvalue and
eigenvector correctors are thus directly given by
\eqref{eq:cvcorrspect} and \eqref{eq:conveig}, respectively.
 
\subsection{Correctors for time dependent problems}
\label{sec:timedep}

As an application of the preceding theory, let us now consider an
evolution problem of the form
\begin{equation}
  \label{eq:evol}
  u_t + \ep Pu =0 ,\quad t>0, \qquad u(0)=u_0,
\end{equation}
where $\ep$ is a constant, typically $\ep=1$ or $\ep=i$, and $P$ is a
symmetric pseudodifferential operator with domain ${\cal D}(P)\subset
L^2(D)$ for some subset $D\subset\Rm^d$ and with a compact inverse
$A=P^{-1}$, which we assume without loss of generality, has positive
eigenvalues.

We then consider the randomly perturbed problem
\begin{equation}
  \label{eq:evolpert}
  u^\eta_t + \ep P_\eta u_\eta =0,\quad t>0,\qquad u_\eta(0)=u_0,
\end{equation}
where $P_\eta(\omega)$ verifies the same hypotheses as $P$ with
compact inverse $A_\eta=P_\eta^{-1}$. 

We assume that $A_\eta$ and $A$ are sufficiently close so that
\eqref{eq:errorAeps} holds. Following the notation of the preceding
section, we denote by $\lambda_n$ and $\lambda^\eta_n$ the eigenvalues
of $A$ and $A_\eta$ and by $u_n$ and $u_n^\eta$ the corresponding
eigenvectors.

We then verify that
\begin{displaymath}
  u(t) = e^{-\ep t P} u_0 = \dsum_n e^{-\ep\lambda_n t} 
   (u_n,u_0) u_n:= \dsum_n \alpha_n(t) u_n,\qquad
   \alpha_n(t) = e^{-\ep\lambda_n t} (u_n,u_0).
\end{displaymath}
and
\begin{displaymath}
  u_\eta(t) = \dsum_n \alpha^\eta_n(t) u^\eta_n,\qquad
    \alpha_n^\eta(t) = e^{-\ep\lambda_n^\eta t} (u_\eta^n,u_0).
\end{displaymath}

We can now compare the Fourier coefficients as follows:
\begin{equation}
  \label{eq:diffFcoef}
  \dfrac{\alpha^\eta_n-\alpha_n}\eta
   = \dfrac{e^{-\ep\lambda_n^\eta t} - e^{-\ep\lambda_n t}}\eta
    (u_n,u_0) + e^{-\ep\lambda_n t} (\dfrac{u^\eta_n-u_n}\eta,u_0)
    + r_\eta, 
\end{equation}
where $|r_\eta|\to0$ strongly in $L^p(\Omega)$ as $\eta\to0$. This may
be recast as
\begin{equation}
  \label{eq:diffFcoef2}
  \dfrac{\alpha^\eta_n-\alpha_n}\eta
   = e^{-\ep\lambda_n t} \ep t 
      \dfrac{\lambda_n -\lambda_n^\eta}\eta
    (u_n,u_0) + e^{-\ep\lambda_n t} (\dfrac{u^\eta_n-u_n}\eta,u_0)
    + s_\eta, 
\end{equation}
where $|s_\eta|\to0$ strongly in $L^p(\Omega)$ as $\eta\to0$.

The above difference thus converges to a mean zero Gaussian random
variable whose variance may easily be estimated from the results
obtained in the preceding section.

Since we do not control the convergence of the eigenvectors for
arbitrary values of $n$ (because we do not control in this study the
speed of convergence in distribution of the random correctors), we
cannot obtain the law of the full corrector $u_\eta(t)-u(t)$. We can,
however, obtain a corrector for the low frequency parts $u_N^\eta(t)$
and $u^N(t)$ of $u_\eta(t)$ and $u(t)$, respectively, where only the
$N$ first terms are kept in the sum in the index $n$. We may easily
estimate the corrector for $(u_N^\eta(t)-u_N(t),u_m)$ using the above
expansion for the Fourier coefficients and the results obtained in the
preceding section. We again obtain that the corrector is a mean zero
Gaussian variable whose variance may be calculated explicitly.

Other time-dependent equations may be treated in a similar way. For
instance, the wave equation
\begin{equation}
  \label{eq:wave}
  u_{tt} + Pu=0,\qquad u(0)=u_0,\quad u_t(0)=g_0,
\end{equation}
where $P$ is a symmetric operator with compact and positive definite
inverse, may be recast as
\begin{equation}
  \label{eq:fos}
  w_t -Aw =0, \quad w(0)=w_0,\qquad w=\left(
    \begin{matrix}
      u\\u_t 
    \end{matrix} \right),\quad
   A=\left(
     \begin{matrix}
       0&1 \\ P&0
     \end{matrix}\right).
\end{equation}
We verify that the eigenvalues $\lambda_n$ of $A$ are purely imaginary
and equal to $\pm i \sqrt{\lambda_P}$, where $\lambda_P$ are the
positive eigenvalues of $P$. The orthogonal projector onto the $n$th
eigenspace of $A$ is found to be
\begin{displaymath}
  \Pi_{A,\lambda} = \left(
    \begin{matrix}
      \Pi_{P,-\lambda^2} & 0 \\ 0 & \lambda \Pi_{P,-\lambda^2}
    \end{matrix} \right),
\end{displaymath}
so that
\begin{equation}
  \label{eq:solwave}
  \left(
    \begin{matrix}
      u\\u_t
    \end{matrix}\right)(t) =
  \dsum_\lambda e^{-\lambda t} \left(
    \begin{matrix}
      \Pi_{P,-\lambda^2} & 0 \\ 0 & \lambda \Pi_{P,-\lambda^2}
    \end{matrix} \right) \left(
    \begin{matrix}
      u_0\\g_0
    \end{matrix}\right).
\end{equation}
A similar expression may be used for the perturbed problem
$u_\eta(t)$, where $P$ is replaced by $P_\eta$. The results presented
earlier in this section easily generalize to provide an estimate for
the low frequency component of $u(t)-u_\eta(t)$. We leave the details
to the reader.

When the Green's function associated to the operator
$\partial_t+\epsilon P$ is sufficiently regular, for instance when
$\epsilon P=-\Delta$, more refined results may be obtained by
considering expansions similar to the expansion \eqref{eq:intHelm}
considered for steady-state problems. We do not consider such
developments here.
\section{Conclusions}
\label{sec:conclu}

We have considered the corrector to the homogenization of source and
spectral problems for the Helmholtz equation with highly oscillatory
random potential. The method works because the operator ${\cal
  G}q_\eps$ appearing in \eqref{eq:intHelm} may be seen as
lower-order, in the sense that it converges rapidly to $0$ with
$\eps$. This requires that the homogenized solution ${\cal G}f$ be a
good approximation to the source problem \eqref{eq:intHelm}. The
method was then generalized to the one-dimensional elliptic problem
\eqref{eq:1dellipt}, which after a change of variables to harmonic
coordinates, may also be recast as an integral equation
\eqref{eq:intueps} with a term ${\cal G}_\eps\tilde q_\eps$ that may
also be seen as lower-order.

Such expansions are not currently available for more challenging
problems of the form $-\nabla\cdot a_\eps(\bx,\omega)\nabla u_\eps=f$,
augmented with appropriate boundary conditions. The use of the Green's
function to the homogenized elliptic equation does not allow for a
rapidly converging expansion of the form \eqref{eq:intHelm} or
\eqref{eq:intueps}. The analysis of correctors for such equations, for
which current state of the art estimations are given in
\cite{Yur-SM-86}, remains an open problem; see \cite{CN-EJP-00} for a
related discretized elliptic equation.

The correctors were analyzed here in the setting where the random
coefficients have integrable correlation function $R(\bx)$ in
\eqref{eq:corrq} (and additional mixing properties). The expansions in
\eqref{eq:intHelm} and \eqref{eq:intueps} may be generalized to random
coefficients with correlation functions $R(\bx)$ which decay as
$|\bx|^{-\alpha d}$ for some $0<\alpha<1$ as $|\bx|\to\infty$. In such
frameworks, following the expansions obtained in \cite{BGMP-AA}, we
expect random correctors with Gaussian statistics and amplitudes of
order $\eps^{\alpha \frac d2}$ rather than $\eps^{\frac d2}$, at least
for dimensions $1\leq d\leq3$ for the Helmholtz problem. These
long-range effects will be analyzed elsewhere.

\section*{Acknowledgment}

The author would like to thank Josselin Garnier, Wenjia Jing, Tomasz
Komorowski, and George Papanicolaou for stimulating discussions on the
subject of equations with random coefficients and central limit
theorems.  This work was supported in part by NSF Grant DMS-0239097
and an Alfred P.  Sloan Fellowship.



\end{document}